\documentclass[11pt]{article}
\usepackage[margin=1in]{geometry}
\usepackage{amsmath,amsthm,amsfonts,amssymb}
\usepackage{graphicx,enumerate}
\graphicspath{ {./figures/} }
\usepackage[title]{appendix}
\usepackage[
            CJKbookmarks=true,
            bookmarksnumbered=true,
            bookmarksopen=true,
            colorlinks=true,
            citecolor=red,
            linkcolor=blue,
            anchorcolor=red,
            urlcolor=blue
            ]{hyperref}

\newcommand{\eps}{\varepsilon}

\usepackage{tikz}
\usepackage{tikz-qtree}

\usepackage{tikz-cd}

\usetikzlibrary{calc,arrows,cd}
\usetikzlibrary{decorations.pathreplacing,decorations.markings}

\tikzstyle{dot}=[circle,fill,black,inner sep=1pt]

\tikzset{
  on each segment/.style={
    decorate,
    decoration={
      show path construction,
      moveto code={},
      lineto code={
        \path [#1]
        (\tikzinputsegmentfirst) -- (\tikzinputsegmentlast);
      },
      curveto code={
        \path [#1] (\tikzinputsegmentfirst)
        .. controls
        (\tikzinputsegmentsupporta) and (\tikzinputsegmentsupportb)
        ..
        (\tikzinputsegmentlast);
      },
      closepath code={
        \path [#1]
        (\tikzinputsegmentfirst) -- (\tikzinputsegmentlast);
      },
    },
  },
  mid arrow/.style={postaction={decorate,decoration={
        markings,
        mark=at position .5 with {\arrow[#1]{stealth}}
      }}},
  early arrow/.style={postaction={decorate,decoration={
        markings,
        mark=at position .2 with {\arrow[#1]{stealth}}
      }}},
}

\tikzstyle{int}=[draw, fill=blue!15, minimum size=2em]
\tikzstyle{init} = [pin edge={to-,thin,black}]

\def\alternatecolorred{%
    \pgfkeysalso{red}%
    \global\let\alternatecolor\alternatecolorblue 
}
\def\alternatecolorblue{%
    \pgfkeysalso{blue}%
    \global\let\alternatecolor\alternatecolorred 
}

\newcommand{\altred}{\let\alternatecolor\alternatecolorred 
\tikzset{every edge/.append code = {%
    \global\let\currenttarget\tikztotarget 
    \pgfkeysalso{append after command={(\currenttarget)}}
			\alternatecolor
}}
}
\newcommand{\altblue}{\let\alternatecolor\alternatecolorblue 
\tikzset{every edge/.append code = {%
    \global\let\currenttarget\tikztotarget 
    \pgfkeysalso{append after command={(\currenttarget)}}
			\alternatecolor
}}
}

\tikzstyle{vertexdot}=[circle, draw, fill=black, minimum size=3,inner sep=0pt]

\newtheorem{theorem}{Theorem}
\newtheorem{lemma}{Lemma}
\newtheorem{proposition}{Proposition}

\usepackage{color}
\usepackage{subfigure}
\usepackage{algorithm}
\usepackage{algorithmic}

\usepackage{xspace,prettyref}

\newcommand{\stepa}[1]{\overset{\rm (a)}{#1}}

\newcommand{\iiddistr}{{\stackrel{\text{\iid}}{\sim}}}
\newcommand{\ones}{\mathbf 1}

\newcommand{\reals}{{\mathbb{R}}}
\newcommand{\integers}{{\mathbb{Z}}}

\newcommand{\bfF}{\mathbf F}

\newcommand{\diff}{{\rm d}}

\newcommand{\Expect}{\mathbb{E}}
\newcommand{\expect}{\mathbb{E}}
\newcommand{\Prob}{\mathbb{P}}

\newcommand{\prob}[1]{ \mathbb{P}\left\{ #1 \right\} }

\newcommand{\var}{\mathsf{var}}
\newcommand{\Cov}{\text{Cov}}

\newcommand{\Bern}{{\rm Bern}}

\newcommand{\ie}{i.e.\xspace}
\newcommand{\iid}{i.i.d.\xspace}
\newrefformat{eq}{(\ref{#1})}
\newrefformat{chap}{Chapter~\ref{#1}}
\newrefformat{sec}{Section~\ref{#1}}
\newrefformat{alg}{Algorithm~\ref{#1}}
\newrefformat{fig}{Fig.~\ref{#1}}
\newrefformat{tab}{Table~\ref{#1}}
\newrefformat{rmk}{Remark~\ref{#1}}
\newrefformat{clm}{Claim~\ref{#1}}
\newrefformat{def}{Definition~\ref{#1}}
\newrefformat{cor}{Corollary~\ref{#1}}
\newrefformat{lmm}{Lemma~\ref{#1}}
\newrefformat{prop}{Proposition~\ref{#1}}
\newrefformat{app}{Appendix~\ref{#1}}
\newrefformat{hyp}{Hypothesis~\ref{#1}}
\newrefformat{thm}{Theorem~\ref{#1}}
\newrefformat{ass}{Assumption~\ref{#1}}
\newrefformat{conj}{Conjecture~\ref{#1}}

\newcommand{\pth}[1]{\left( #1 \right)}
\newcommand{\qth}[1]{\left[ #1 \right]}
\newcommand{\sth}[1]{\left\{ #1 \right\}}

\newcommand{\norm}[1]{\left\|{#1} \right\|}

\newcommand{\lnorm}[2]{\left\|{#1} \right\|_{{#2}}}

\newcommand{\Fnorm}[1]{\lnorm{#1}{\rm F}}
\newcommand{\fnorm}[1]{\|#1\|_{\rm F}}
\newcommand{\opnorm}[1]{\left\| #1 \right\|_{\rm op}}
\newcommand{\Opnorm}[1]{\| #1 \|_{\rm op}}
\newcommand{\iprod}[2]{\left \langle #1, #2 \right\rangle}
\newcommand{\Iprod}[2]{\langle #1, #2 \rangle}
\newcommand{\indc}[1]{{\mathbf{1}_{\left\{{#1}\right\}}}}

\newcommand{\diag}{\mathsf{diag}}

\newcommand{\tA}{{\widetilde{A}}}
\newcommand{\tB}{{\widetilde{B}}}

\newcommand{\tH}{{\widetilde{H}}}

\newcommand{\tJ}{{\widetilde{J}}}

\newcommand{\tT}{{\widetilde{T}}}

\newcommand{\tX}{{\widetilde{X}}}
\newcommand{\tY}{{\widetilde{Y}}}
\newcommand{\tZ}{{\widetilde{Z}}}

\newcommand{\calE}{{\mathcal{E}}}

\newcommand{\calN}{{\mathcal{N}}}
\newcommand{\calO}{{\mathcal{O}}}
\newcommand{\calP}{{\mathcal{P}}}
\newcommand{\calQ}{{\mathcal{Q}}}

\newcommand{\frakc}{{\mathfrak{c}}}

\DeclareMathAlphabet{\varmathbb}{U}{bbold}{m}{n}

\newcommand{\argmax}{\mathrm{argmax}}

\renewcommand{\hat}{\widehat}
\renewcommand{\tilde}{\widetilde}

\usepackage{bbm}

\newcommand{\R}{\mathbb{R}}
\newcommand{\C}{\mathbb{C}}

\DeclareMathOperator{\Tr}{Tr}

\newcommand{\vecc}{\mathsf{vec}}

\newcommand{\overlap}{\mathsf{overlap}}

\newcommand{\ii}{\mathrm{i}}

\newcommand{\fS}{\mathfrak{S}} 

\newcommand{\ER}{Erd\H{o}s-R\'{e}nyi\xspace}

\begin{document}

\pgfdeclarelayer{background}
\pgfdeclarelayer{foreground}
\pgfsetlayers{background,main,foreground}

\title{Random Graph Matching in Geometric Models: the Case of Complete Graphs}

\author{Haoyu Wang, Yihong Wu, Jiaming Xu, and Israel Yolou\thanks{
H.\ Wang is with the Department of Mathematics, Yale University, New Haven, USA, \texttt{haoyu.wang@yale.edu}.
Y.\ Wu is with the Department of Statistics and Data Science, Yale University, New Haven, USA, \texttt{yihong.wu@yale.edu}.
J.\ Xu is with The Fuqua School of Business, Duke University, Durham NC, USA, \texttt{jx77@duke.edu}.
I.\ Yolou is with the Departments of Mathematics and Computer Science, Yale University, New Haven, USA, \texttt{israel.yolou@yale.edu}.
}}

\maketitle

\begin{abstract}

This paper studies the problem of matching two complete graphs with edge weights correlated through latent geometries, extending a recent line of research on random graph matching with independent edge weights to geometric models. 
Specifically, given a random permutation $\pi^*$ on $[n]$ and $n$ iid pairs of correlated Gaussian vectors $\{X_{\pi^*(i)}, Y_i\}$ in $\reals^d$ with noise parameter $\sigma$, the edge weights are given by $A_{ij}=\kappa(X_i,X_j)$ and $B_{ij}=\kappa(Y_i,Y_j)$
for some link function $\kappa$. The goal is to recover the hidden vertex correspondence $\pi^*$ based on the observation of $A$ and $B$.
We focus on the dot-product model with $\kappa(x,y)=\langle x, y \rangle$ and Euclidean distance model with $\kappa(x,y)=\|x-y\|^2$, in the low-dimensional regime of $d=o(\log n)$ wherein the underlying geometric structures are most evident. We derive an approximate maximum likelihood estimator, which provably achieves, with high probability, perfect recovery of $\pi^*$  when $\sigma=o(n^{-2/d})$ and almost perfect recovery with a vanishing fraction of errors when $\sigma=o(n^{-1/d})$. 
Furthermore, these conditions are shown to be information-theoretically optimal even when the latent coordinates $\{X_i\}$ and $\{Y_i\}$ are observed, complementing the recent results of \cite{dai2019database} and \cite{kunisky2022strong} in geometric models of the planted bipartite matching problem. 
As a side discovery, we show that the celebrated  spectral algorithm of \cite{umeyama1988eigendecomposition} emerges as a further approximation to the maximum likelihood  in the geometric model.

\end{abstract}

\section{Introduction}
	\label{sec:intro}
	

	Graph matching (or network alignment) refers to finding the best vertex correspondence between two graphs that maximizes the total number of common edges.
	While this problem, as an instance of quadratic assignment problem, is computationally intractable in the worst case, significant headways, both information-theoretic and algorithmic, have been achieved in the average-case analysis under meaningful statistical models \cite{cullina2016improved,cullina2017exact,ding2018efficient,barak2019nearly,FMWX19a,FMWX19b,Hall2020partial,wu2021settling,ganassali2020tree,ganassali2021correlation,mao2021random,mao2021exact}.
One of the most popular models is the \emph{correlated \ER graph} model~\cite{pedarsani2011privacy}, where both observed graphs are \ER graphs with edges correlated 
	through a latent vertex matching; more generally, in the \emph{correlated Wigner} model, the observations are two weighted graph with correlated edge weights (e.g.~Gaussians \cite{ding2018efficient,dai2019database,FMWX19a,ganassali2020sharp}).
	Despite their simplicity, these models inspired a number of new algorithms that achieve strong performance both theoretically and practically \cite{ding2018efficient,FMWX19a,FMWX19b,ganassali2020tree,ganassali2021correlation,mao2021random,mao2021exact}. Nevertheless, one of the major limitations of models with independent edges is that they fail to capture graphs with spatial structure~\cite{armiti2014geometric}, such as those arising in computer vision datasets (e.g.~mesh graphs obtained by triangulating 3D shapes \cite{lahner2016shrec}).
	In contrast to \ER-style model with iid edges, \emph{geometric graph models}, such as random dot-product graphs and random geometric graphs, take into account the latent geometry by embedding each node in a Euclidean space and determines edge connection between two nodes by the proximity of their geographical location. 
	While the coordinates are typically assumed to be independent (e.g.~Gaussians or uniform over spheres or hypercubes), the edges or edge weights are now dependent.
	The main objective for the present paper is to study graph matching in correlated geometric graph models, where the network correlation is due to that of the latent coordinates.
	
	\subsection{Model}
	\label{sec:model}
	
	Given two point clouds $\{X_1,\ldots,X_n\}$ and $\{Y_1,\ldots,Y_n\}$ in $\reals^d$, we construct 
	two weighted graphs on the vertex set $[n]$ with weighted adjacency matrices $A$ and $B$ as follows.
	For each $i,j$, let $A_{ij} \overset{\rm ind}{\sim} W(\cdot|X_i,X_j)$ and $B_{ij} \overset{\rm ind}{\sim} W(\cdot|Y_i,Y_j)$, 
	for some probability transition kernel $W$.
	The coordinates are correlated through a latent matching as follows:
Consider a Gaussian model
\[
Y_i= X_{\pi^*(i)} + \sigma Z_i, \quad i=1,\ldots,n,
\]
where $X_i, Z_i$'s are iid $\calN(0, I_d)$ vectors and $ \pi^* $ is uniform on $ S_n $, the set of all permutations on $[n]$.
In matrix form, we have
\begin{equation}
Y=\Pi^* X + \sigma Z,
\label{eq:model-LAP}
\end{equation}
where $ X,Y,Z \in \R^{n \times d} $ are matrices whose rows are $ X_i $'s, $ Y_i $'s and $ Z_i $'s respectively, $ \Pi^* \in \fS_n$ denotes the permutation matrix corresponding to $ \pi^* $, and $\fS_n$ is the collection of all permutation matrices.
Given the observation $A$ and $B$, the goal is to recover the latent correspondence $\pi^*$.
	
	Of particular interest are the following special cases:
	\begin{itemize}
	
		\item \emph{Dot-product model}: The observations are complete graphs with pairwise inner products as edge weights, namely, $A_{ij} = \Iprod{X_i}{X_j}$ and $B_{ij} = \Iprod{Y_i}{Y_j}$. As such, the weighted adjacency matrices are $ A=XX^\top $ and $ B=YY^\top $, both Wishart matrices. 
		It is clear that from $A$ and $B$ one can reconstruct $X$ and $Y$ respectively, each up to a global orthogonal transformation on the rows.
		In this light, the model is also equivalent to the so-called \emph{Procrustes Matching} problem \cite{maron2016point,dym2017exact,grave2019unsupervised}, where $Y$ in \prettyref{eq:model-LAP} undergoes a further random orthogonal transformation --
		see \prettyref{app:related} for a detailed discussion.

		\item \emph{Distance model}: The edge weights are pairwise squared distances $A_{ij} = \|X_i-X_j\|^2$ and $B_{ij} = \|X_i-X_j\|^2$. 
		This setting corresponds to the classical problem of multi-dimensional scaling (MDS), where the goal is to reconstruct the coordinates (up to global shift and orthogonal transformation) from the distance data (cf.~\cite{borg2005modern}).

		\item \emph{Random Dot Product Graph (RDPG)}: 
		In this model, the observed data are two graphs with adjacency matrices $A$ and $B$, 
		where $A_{ij} \overset{\rm ind}{\sim} \Bern\left(\kappa(\iprod{X_i}{X_j})\right)$ and $B_{ij} \overset{\rm ind}{\sim} \Bern\left(\kappa(\iprod{X_i}{X_j})\right)$ conditioned on $X$ and $Y$, and 
		$\kappa:\reals\to [0,1]$ is some link function, e.g.~$\kappa(t)=e^{-t^2/2}$. 
		In this way, we observe two instances of RDPG that are correlated through the underlying points and the latent matching.
See \cite{athreya2017statistical} for a recent survey on RDPG.

		%
		
		\item \emph{Random Geometric Graph (RGG)}: 
		Similar to RDPG,  $A_{ij} \overset{\rm ind}{\sim} \Bern(\kappa(\|X_i-X_j\|))$ 
		conditioned on $X_1,\ldots,X_n$ for some link function $\kappa:\reals_+\to [0,1]$ applied to the pairwise distances.
		The second RGG instance $B$ is constructed in the same way using $Y_1,\ldots,Y_n$.
		A simple example is $\kappa(t) = \indc{t \leq r} $ for some threshold $r>0$, where each pair of points within distance $r$ is connected \cite{gilbert1961random}; 
		see the monograph \cite{penrose2003random} for a comprehensive discussion on RGG.
	\end{itemize}
	
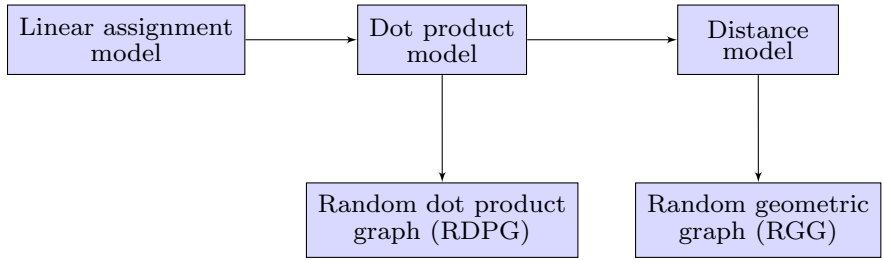
\begin{figure}[ht]
	\centering
\begin{tikzpicture}[scale=1.2,transform shape,node distance=2.5cm,auto,>=latex']
    \node [int] (a) {$\substack{\textrm{Linear assignment} \\ \textrm{model}}$};   
    \node [int] (c) [right of=a,node distance=3.5cm] {$\substack{\textrm{Dot product} \\ \textrm{model}}$}; 
		\node [int] (d) [right of=c,node distance=3.5cm] {$\substack{\textrm{~~Distance~~} \\ \textrm{model}}$}; 
		\node [int] (e) [below of=c,node distance=2cm] {$\substack{\textrm{Random dot product} \\ \textrm{graph (RDPG)}}$}; 
		\node [int] (f) [below of=d,node distance=2cm] {$\substack{\textrm{Random geometric} \\ \textrm{graph (RGG)}}$}; 
    \path[->] (a) edge node {}  (c);
		\path[->] (c) edge node {}  (d);
		\path[->] (c) edge node {}  (e);
		\path[->] (d) edge node {}  (f);
\end{tikzpicture}
\label{fig:model}
\caption{Geometric matching models. Here arrows denote statistical ordering.}
\end{figure}


Let us mention that the model where the two point clouds are directly observed has been recently studied by~\cite{dai2019database,dai2020achievability} in the context of feature matching and independently by~\cite{kunisky2022strong} as a geometric model for the planted matching problem, extending the previous work in \cite{Chertkov2010,moharrami2021planted,DWXY21} with iid weights to a geometric (low-rank) setting.
	In this model, $X$ and $Y$ in \prettyref{eq:LAP} are observed and the maximum likelihood estimator (MLE) of $\pi^*$ amounts to solving 
	\begin{equation}
	\max_{\Pi \in \fS_n} \Iprod{Y}{\Pi X}.
	\label{eq:LAP}
	\end{equation}
	which is a linear assignment (max-weight matching) problem on the weighted complete bipartite graph with weight matrix $YX^\top$.
	In the sequel we shall refer to this setting as the linear assignment model, 
	which we also study in this paper for the sake of proving impossibility results for the more difficult graph matching problem, as the coordinates are latent and only \emph{pairwise information} are available.



	\prettyref{fig:model} elucidates the logical connections between the aforementioned models. Among these, linear assignment model is the most informative, followed by the dot product model and the distance model, whose further stochastically degraded versions are RDPG and RGG, respectively. 
	As a first step towards understanding graph matching on geometric models, in this paper we study the case of weighted complete graphs in the dot product and distance models.


\subsection{Main results}
\label{sec:main}

By analyzing the MLE \prettyref{eq:LAP} in the stronger linear assignment model \prettyref{eq:model-LAP}, \cite{kunisky2022strong} identified a critical scaling of dimension $d$ at $\log n$: 
\begin{itemize}
	\item In the low-dimensional regime of $d\ll \log n$, accurate reconstruction requires the noise level $\sigma$ to be vanishingly small. More precisely, 
	with high probability, the MLE \prettyref{eq:LAP} recovers the latent $\pi^*$ perfectly
	 (resp.~with a vanishing fraction of errors) provided that $\sigma = o(n^{-2/d})$ (resp.~$\sigma = o(n^{-1/d})$).

	\item In the high-dimensional regime of $d \gg \log n$, it is possible for  $\sigma^2$ to be as large as $\frac{d}{(4+o(1))\log n}$. Since the dependency between the edges weakens as the latent dimension increases,\footnote{For the Wishart matrix, it is known \cite{jiang2015approximation,bubeck2018entropic} that the total variation between the joint law of the off-diagonals and their iid Gaussian counterpart converges to zero provided that $d=\omega(n^3)$. Analogous results have also been obtained in \cite{bubeck2016testing} showing that high-dimensional RGG is approximately \ER.}
this is consistent with the known results in the correlated \ER and Wigner model.
For example, to match two GOE matrices with correlation coefficient $\rho$, the sharp reconstruction threshold is at 
$\rho^2 = \frac{(4+o(1))\log n}{n}$ \cite{ganassali2021sharp,wu2021settling}. 
\end{itemize}
In this paper we mostly focus on the low-dimensional setting as this is the regime where geometric graph ensembles are structurally distinct from \ER graphs. Our main findings are two-fold: 
\begin{enumerate}
	\item The same reconstruction thresholds 
	remain achievable even when the coordinates are latent and only inner-product or distance data are accessible.
	
	\item Furthermore, these thresholds cannot be improved even when the coordinates are observed.
\end{enumerate}

To make these results precise, we start with the dot-product model with $A=XX^\top$ and $B=YY^\top$, and $Y=\Pi^*X+\sigma Z$ according to \prettyref{eq:model-LAP}. In this case the MLE turns out to be much more complicated than \prettyref{eq:LAP} for the linear assignment model. 
As shown in \prettyref{app:MLE}, the MLE takes the form
\begin{equation}
\hat{\Pi}_{\mathrm{ML}} = 
\arg\max_{\Pi\in\fS_n}  \int_{O(d)} \diff Q \exp\pth{\frac{\Iprod{B^{1/2}}{\Pi A^{1/2} Q}}{\sigma^2}},
\label{eq:MLE-dotprod}
\end{equation}
where the integral is with respect to the Haar measure on the orthogonal group $O(d)$, $A^{1/2} \triangleq U \Lambda^{1/2}\in\reals^{n\times d}$ based on the SVD
$A = U \Lambda U^\top$, and similarly for $B^{1/2}$.
It is unclear whether the above Haar integral has a closed-form solution,\footnote{The integral in \prettyref{eq:MLE-dotprod} can be reduced to computing 
$\int \diff Q \exp(\Iprod{\Lambda}{Q})$ for a diagonal $\Lambda$, which, in 
 principle, can be evaluated by Taylor expansion and applying formulas for the joint moments of $Q$ in \cite[Theorem 2.2]{matsumoto2013weingarten}.} let alone how to optimize it over all permutations. Next, we turn to its approximation.

As we will show later, in the low-dimensional case of $d=o(\log n)$, meaningful reconstruction of the latent matching is information-theoretically impossible unless $\sigma$ vanishes with $n$ at a certain speed. 
In the regime of small $\sigma$, Laplace's method suggests that the predominant contribution to the integral in \prettyref{eq:MLE-dotprod} 
comes from the maximum $\Iprod{B^{1/2}}{\Pi A^{1/2} Q}$ over $Q \in O(d)$. 
Using the dual form of the nuclear norm $\|X\|_* = \max_{Q \in O(d)} \Iprod{X}{Q}$, where $\|X\|_*$ denotes the sum of all singular values of $X$, we arrive at the following approximate MLE:
\begin{equation}
\hat{\Pi}_{\mathrm{AML}} = \arg\max_{\Pi \in \fS_n} \|(A^{1/2})^\top \Pi^\top B^{1/2}\|_*.
\label{eq:MLEapprox}
\end{equation}
We stress that the above approximation to the MLE \prettyref{eq:MLE-dotprod} is justified for the low-dimensional regime where $\sigma$ is small. In the high-dimensional (high-noise) case, the approximate MLE actually takes on the form of a quadratic assignment problem (QAP), which is the MLE for the well-studied iid model \cite{cullina2016improved}; in the special case of the dot-product model, it amounts to replacing the nuclear norm in \prettyref{eq:MLEapprox} by the Frobenius norm. We postpone this discussion to \prettyref{sec:discuss}.

To measure the accuracy of a given estimator $\hat \pi$, we define
$$ \overlap(\hat{\pi},\pi) \triangleq \frac{1}{n} |\sth{ i \in [n]:\hat{\pi}(i) = \pi(i) }| $$
as the fraction of nodes whose matching is correctly recovered.
The following result identifies the threshold at which the approximate MLE achieves perfect or almost perfect recovery.

\begin{theorem}[Recovery guarantee of AML in the dot-product model]
\label{thm:main}
Assume the dot-product model with $d=o(\log n)$. Let $ \hat{\pi}_{\mathrm{AML}} $ be the approximate MLE defined in \prettyref{eq:MLEapprox}.
\begin{enumerate}
\item[(i)] If $ \sigma \ll n^{-2/d} $, the estimator $ \hat{\pi}_{\mathrm{AML}} $ achieves perfect recovery with high probability:
\begin{equation}
\prob{ \overlap(\hat{\pi}_{\mathrm{AML}},\pi^*)=1  } = 1-o(1).
\label{eq:main1}
\end{equation}
\item[(ii)] 
If $ \sigma \ll n^{-1/d} $, the estimator $ \hat{\pi}_{\mathrm{AML}} $ achieves almost perfect recovery with high probability:
\begin{equation}
\prob{ \overlap(\hat{\pi}_{\mathrm{AML}},\pi^*) \geq 1-o(1)  } = 1-o(1).
\label{eq:main2}
\end{equation}
\end{enumerate}
\end{theorem}

A few remarks are in order:

\begin{itemize}
    \item In fact we will show the following nonasymptotic estimate that implies \prettyref{eq:main2}: For all sufficiently small $ \eps$, if $ \sigma^{-d} > 16n 2^{2/\eps} $, then
$\overlap(\hat{\pi}_{\mathrm{AML}},\pi^*) \geq 1-\eps $ with probability tending to one.

\item The estimator \prettyref{eq:MLEapprox}
has previously appeared in the literature of Procrustes matching \cite{grave2019unsupervised}, albeit not as an approximation to the MLE in a generative model. See \prettyref{app:related} for a detailed discussion.
    
    \item Unlike linear assignment, it is unclear how to solve the optimization in \prettyref{eq:MLEapprox} over permutations efficiently.
    Nevertheless, for constant $d$ we show that it is possible to find an approximate solution in time that is polynomial in $n$ that achieves the same statistical guarantee as in \prettyref{thm:main}.
Indeed, note that \prettyref{eq:MLE-dotprod} is equivalent to the double maximization
\begin{equation}
\hat{\Pi}_{\mathrm{AML}} = \arg\max_{\Pi \in \fS_n} \max_{Q\in O(d)} \Iprod{B^{1/2}}{\Pi A^{1/2} Q}.
\label{eq:MLEapprox1}
\end{equation}
Approximating the inner maximum over a suitable discretization of $O(d)$, each maximization over $\Pi$ for fixed $Q$ is a linear assignment problem, which can be solved in $O(n^3)$ time. In \prettyref{sec:exp}, we provide a heuristic that shows \prettyref{eq:MLEapprox1} can be further approximated by the classical spectral algorithm of Umeyama \cite{umeyama1988eigendecomposition} which is much faster in practice and achieves good empirical performance.
For $d$ that grows with $n$, it is an open question to find a polynomial-time algorithm that attains the (optimal, as we show next) threshold in \prettyref{thm:main}.
\end{itemize}

Next, we proceed to the more difficult distance model, where $A_{ij} = \|X_i-X_j\|^2$ and $B_{ij} = \|Y_i-Y_j\|^2$.
Deriving the exact MLE in this model appears to be challenging; instead, 
we apply the estimator \prettyref{eq:MLEapprox} to an appropriately centered version of the data matrices.
Let $\ones\in\reals^n$ denotes the all-one vector and define 
$\bfF=\frac{1}{n}\ones\ones^\top$.
Then $A = -2 XX^\top + a\ones^\top+\ones a^\top$ and 
$B = -2 YY^\top + b\ones^\top+\ones b^\top$, where $a=(\|X_i\|^2)$ and $b=(\|Y_i\|^2)$.
Strictly speaking, the vectors $a$ and $b$ are correlated with the ground truth $\pi^*$, since $b$ can be viewed as a noisy version of $\Pi^*a$; however, we expect them to inform very little about $\pi^*$ because such scalar-valued observations are highly sensitive to noise (analogous to degree matching in correlated \ER graphs~\cite[Section 1.3]{ding2018efficient}).
As such, we ignore $a$ and $b$ by projecting $A$ and $B$ to the orthogonal complement of the vector $\ones$. 
Specifically, we compute, as commonly done in the MDS literature (see e.g.~\cite{shang2003localization,oh2010sensor}), 
\begin{equation}
\tilde A = - \frac{1}{2} (I-\bfF)A(I-\bfF), \quad 
\tilde B = - \frac{1}{2} (I-\bfF)B(I-\bfF).
\end{equation}
 It is easy to verify that $\tilde A = \tilde X \tilde X^\top$ and $\tilde B = \tilde Y \tilde Y^\top$, where $\tilde X=(I-\bfF)X$ and $\tilde Y=(I-\bfF)Y$ consist of centered coordinates $\tilde X_i = X_i-\bar X$ and $\tilde Y_i = Y_i-\bar Y$ respectively, with $\bar X = \frac{1}{n} \sum_{i=1}^n X_i$ and $\bar Y = \frac{1}{n} \sum_{i=1}^n Y_i$. Overall, we have reduced the distance model to a dot product model where the latent coordinates are now centered.

One can show that the MLE of $\Pi^*$ given the reduced data $(\tilde A,\tilde B)$ is of the same Haar-integral form \prettyref{eq:MLE-dotprod}. Using again the small-$\sigma$ approximation, 
we arrive at the following estimator by applying \prettyref{eq:MLEapprox} to the centered data $\tilde A$ and $\tilde B$:
\begin{equation}
\tilde{\Pi}_{\mathrm{AML}} = \arg\max_{\Pi \in \fS_n} \|(\tilde A^{1/2})^\top \Pi^\top \tilde B^{1/2}\|_*.
\label{eq:MLEapprox-distance}
\end{equation}

\begin{theorem}[Recovery guarantee in the distance model]
\label{thm:distance}
Assuming the distance model, \prettyref{thm:main} holds under the same condition on $d$ and $\sigma$, with 
the estimator $\tilde{\Pi}_{\mathrm{AML}}$ in \prettyref{eq:MLEapprox-distance} replacing 
$\hat{\Pi}_{\mathrm{AML}}$ in \prettyref{eq:MLEapprox}.
\end{theorem}

Finally, we state an impossibility result for the linear assignment model, proving that the perfect
and almost perfect recovery threshold
of $\sigma =o(n^{-2/d})$ and 
$\sigma =o(n^{-1/d})$ 
obtained by analyzing the MLE in \cite{kunisky2022strong} are in fact information-theoretically necessary.
Complementing \prettyref{thm:main} and \prettyref{thm:distance}, this result also establishes the optimality of the estimator \prettyref{eq:MLEapprox} and \prettyref{eq:MLEapprox-distance} for their respective model.


\begin{theorem}[Impossibility result in the linear assignment model]
\label{thm:opt}
Consider the linear assignment model with $d=o(\log n)$.
\begin{enumerate}[(i)]
 \item  If there exists an estimator that achieves perfect recovery with high probability, then $\sigma \le n^{-2/d}$. \label{opt1}
 
 \item If there exists an estimator that achieves almost perfect recovery with high probability, then $\sigma \le n^{-\left(1-o(1)\right) /d}$. \label{opt2}
\end{enumerate}
Furthermore, in the special case of $d=\Theta(1)$, necessary conditions in $(i)$ and $(ii)$
     can be improved to $\sigma \le o(n^{-2/d})$ and
     $\sigma \le o(n^{-1/d})$, respectively.
\end{theorem}

\prettyref{thm:opt}\eqref{opt1} slightly improves the necessary condition for perfect recovery in \cite{kunisky2022strong} from $\sigma=O(n^{-2/d})$  to $\sigma=o(n^{-2/d})$. 
For almost perfect recovery, the negative result in \cite{kunisky2022strong} is limited to MLE, while \prettyref{thm:opt} holds for all algorithms. Moreover, the necessary condition in  \prettyref{thm:opt}\eqref{opt2} was conjectured in \cite[Conjecture 1.4, item 1]{kunisky2022strong}, which we now resolve in the positive. Finally, while our focus is in the low-dimensional case of $d=o(\log n)$, we also provide necessary conditions that hold for general $d$. (See~\prettyref{app:lower_bounds} for details). 

In view of \prettyref{fig:model}, since the negative results in \prettyref{thm:opt} are proved for the strongest model and 
the positive results in \prettyref{thm:distance} are for the weakest model, we conclude that for all three models, namely, linear assignment, dot-product, and distance model, the thresholds for exact and almost perfect reconstruction is given by $n^{-2/d}$ and $n^{-1/d}$, respectively.

\section{Outline of proofs}

\subsection{Positive results}
\label{sec:positve-sketch}

 The positive results of \prettyref{thm:main} and \prettyref{thm:distance} are proved in Appendix \ref{sec:positve} and \prettyref{app:distance}.
Here we briefly describe the proof strategy in the dot product model.
Suppose we want to bound the probability that the approximate MLE $\hat{\Pi}_{\mathrm{AML}}$ in \prettyref{eq:MLEapprox} makes more than  $t$ number of errors. Denote by $ \diff(\pi_1,\pi_2) \triangleq \sum_{i=1}^n \indc{\pi_1(i) \neq \pi_2(i)} $ the Hamming distance between two permutations $ \pi_1,\pi_2 \in S_n $.
Without loss of generality, we will assume that $ \pi^*=\mathrm{Id} $.
By the orthogonal invariance of $\|\cdot\|_*$, we can assume, \emph{for the sake of analysis}, that $A^{1/2}=X$ and $B^{1/2} = Y$. 
Applying \prettyref{eq:MLEapprox1}, 
\begin{align}
      \prob{ \diff(\hat{\Pi}_{\mathrm{AML}},\mathrm{Id}) > t }  
\leq & \prob{ \max_{\pi: \diff(\pi,\mathrm{Id}) > t} \|X^\top \Pi^\top Y\|_* \geq \|X^\top Y\|_* } \nonumber \\
\leq & \prob{ \max_{\pi: \diff(\pi,\mathrm{Id}) > t} \max_{Q \in O(d)} \Iprod{X^\top \Pi^\top Y}{Q} \geq \Iprod{X^\top Y}{I_d} }.
\label{eq:error0}
\end{align}
For each fixed $\Pi$ and $Q$, averaging over the noise yields, for some absolute constant $c_0$,
\begin{equation}
 \prob{\Iprod{X^\top \Pi^\top Y}{Q} \geq \Iprod{X^\top Y}{I_d} } 
 \leq   \expect \exp \left\{ -\frac{c_0}{ \sigma^2} \Fnorm{X-\Pi X Q}^2 \right\}.
 \label{eq:error}
\end{equation}

In the remaining argument, there are three places where the structure of the orthogonal group $O(d)$ plays a crucial role:
\begin{enumerate}

\item The quantity in \prettyref{eq:error} turns out to depend on $\Pi$ through its cycle type 
and on $Q$ through its eigenvalues. 
Crucially, the eigenvalues of an orthogonal matrix $Q$ lie on the unit circle, denoted by $(e^{\ii \theta_1},\ldots,e^{\ii \theta_d})$, with $|\theta_\ell| \leq \pi$. We then show that the error probability in \prettyref{eq:error} can be further bounded by, for some absolute constant $C_0$,
\begin{equation}
\label{eq:error1}
(C_0 \sigma)^{d(n-\frakc)} 
\pth{\prod_{\ell=1}^d \frac{C_0 \sigma}{\sigma+ |\theta_\ell|}}^{n_1},
\end{equation}
where $n_1$ is the number of fixed points in $\pi$ and $\frakc$ is the total number of cycles.

\item 
In order to bound \prettyref{eq:error0}, we take a union bound over $\pi$ and another union bound over an appropriate discretization of $O(d)$. This turns out to be much subtler than the usual $\delta$-net-based argument, as one needs to implement a localized covering and take into account the local geometry of the orthogonal group. Specifically, note that the error probability in \prettyref{eq:error} becomes larger when $\pi$ is near $\mathrm{Id}$ and when $Q$ is near $I_d$ (i.e.~the phases $|\theta_\ell|$'s are small); fortunately, the entropy (namely, the number of such $\pi$ and such $Q$ within a certain resolution) also becomes smaller, balancing out the deterioration in the probability bound. 
 This is the second place where the structure of $O(d)$ is used crucially, as the local metric entropy of $O(d)$ in the vicinity of $I_d$ is much lower than that elsewhere.

\item
Controlling the approximation error of the nuclear norm is another key step. Note that for any matrix norm of the dual form $\|A\| = \sup_{\|Q\|' \leq 1} \Iprod{A}{Q}$, where $\|\cdot\|'$ is the dual norm of $\|\cdot\|$,  the standard $ \delta $-net argument (cf.~\cite[Lemma 4.4.1]{Vershynin-HDP}) yields a multiplicative approximation $ \max_{Q \in N} \iprod{A}{Q} \geq (1-\delta) \|A\|$,
where $N$ is any $\delta$-net of the dual norm ball.
In general, this result cannot be improved (e.g.~for Frobenius norm); nevertheless, for the special case of nuclear norm, this approximation ratio can be improved from $1-\delta$ to $1-\delta^2$,
as the following result of independent interest shows. This improvement turns out to be crucial for obtaining the sharp threshold.
\begin{lemma}\label{lem:Net_Error}
Let $ N \subset O(d) $ be a $ \delta $-net in operator norm of the orthogonal group $ O(d) $. For any $ A \in \R^{d \times d} $,
\begin{equation}
\max_{Q \in N} \iprod{A}{Q} \geq \left( 1-\frac{\delta^2}{2} \right) \|A\|_*.
\end{equation}
\end{lemma}

\end{enumerate}
The proof of \prettyref{thm:main} is completed by combining \prettyref{eq:error1} 
with a union bound over a specific discretization of $O(d)$, 
whose cardinality satisfies the desired eigenvalue-based local entropy estimate, followed by a union bound over $\pi$ which can be controlled using moment generating function of the number of cycles in a random derangement.


\subsection{Negative results}
\label{sec:negative-sketch}

The information-theoretic lower bounds in  \prettyref{thm:opt} for the linear assignment model  are proved in~\prettyref{app:lower_bounds}. Here we sketch the main ideas. We first derive a necessary condition for almost perfect recovery that holds for any $d$ via a simple mutual information argument~\cite{HajekWuXu_one_info_lim15}:
On one hand, the mutual information $I(\pi^*;X,Y)$ can be upper bounded by the Gaussian channel capacity as $\frac{nd}{2} \log (1+\sigma^{-2})$. On the other hand, to achieve almost perfect recovery, $I(\pi^*;X,Y)$ needs be asymptotically equal to the full entropy $H(\pi^*)$ which is $(1-o(1)) \log n$. These two assertions together immediately imply that $\frac{nd}{2} \log (1+\sigma^{-2}) \ge ((1-o(1)) \log n$, which further simplifies to $\sigma=n^{-(1-o(1))/d}$ when $d=o(\log n)$. However, for constant $d$, this necessary condition turns out to be loose and the main bulk of our proof is to improve it to the optimal condition $\sigma=o(n^{-1/d})$. To this end, we follow the program recently developed in~\cite{DWXY21} in the context of the planted matching model by analyzing the posterior measure of the latent $\pi^*$ given the data $(X,Y)$. 

To start, a simple yet crucial observation in~\cite{DWXY21} is that to prove the impossibility of almost perfect recovery, it suffices to show a random permutation sampled from the posterior distribution is at Hamming distance $\Omega(n)$ away from the ground truth with constant probability. As such, it suffices to show there is more posterior mass over the bad permutations (those
far away from the ground truth) than that over the good permutations (those near the ground truth) in the posterior distribution. To proceed, 
we first bound from above the total posterior mass of good permutations by a truncated first moment calculation applying the large deviation analysis developed in the  proof of the positive results. To bound from below the posterior mass of bad permutations, 
we aim to construct exponentially many bad permutations $\pi$ whose log likelihood $L(\pi)$ is no smaller than $L(\pi^*)$. A key observation is that $L(\pi)-L(\pi^*)$ can be decomposed according to the orbit decomposition of $(\pi^*)^{-1}\circ \pi$:
\begin{align}
    L(\pi)- L(\pi^*)= \frac{1}{\sigma^2} \iprod{\Pi X- \Pi^*X}{Y} =\frac{1}{\sigma^2} \sum_{O \in \calO} \Delta (O), \label{eq:L_diff_decomp}
\end{align}
where $\calO$ denotes the set of orbits in $(\pi^*)^{-1}\circ \pi$ and for any orbit $O=(i_1, i_2, \ldots, i_t)$, 
\begin{align}
\Delta(O) \triangleq 
\sum_{k=1}^t  \iprod{X_{\pi^*(i_{k+1})} -X_{\pi^*(i_k)} }{Y_{i_k}}. \label{eq:Delta_O_def}
\end{align}
Thus, the goal is to find  
a collection of vertex-disjoint orbits $O$ whose total lengths add up to $\Omega(n)$ and each of which is \emph{augmenting} in the sense that $\Delta(O) \ge 0$. Here, a key difference to \cite{DWXY21} is that in the planted matching model with independent edge weights studied there, short augmenting orbits are insufficient to meet the $\Omega(n)$ total length requirement; instead, \cite{DWXY21} resorts to a sophisticated two-stage process that first finds many augmenting paths then connects then into long cycles. Fortunately,
for the linear assignment model in low dimensions of $d=\Theta(1)$, as also observed in~\cite{kunisky2022strong} in their analysis of the MLE, it suffices to look for augmenting $2$-orbits and take their disjoint unions. More precisely, we show that there are $\Omega(n)$ many vertex-disjoint augmenting $2$-orbits. 
This has already been done in~\cite{kunisky2022strong} using a  second-moment method enhanced by an additional concentration inequality. It turns out that the correlation among the augmenting $2$-orbits is mild enough so that a much simpler argument via a basic second-moment calculation followed by an application of Tur\'an's theorem suffices to extract a large vertex-disjoint subcollection. Finally, these vertex-disjoint augmenting $2$-orbits give rise to exponentially many permutations that differ from the ground truth by $\Omega(n)$. 

Finally, we briefly remark on perfect recovery, for which it suffices to focus on the MLE \prettyref{eq:LAP} which minimizes the error probability for uniform $\pi^*$.
In view of the likelihood decomposition given in~\prettyref{eq:L_diff_decomp}, it further suffices to prove the existence of \emph{an} augmenting $2$-orbit. This can be easily done using the second-moment method. A similar strategy was adopted in~\cite{dai2019database}, but our first-moment and second-moment estimates are tighter and hence yield nearly optimal conditions.

\section{Experiments}
\label{sec:exp}
In this section we present preliminary numerical results on synthetic data from the dot product model. 
As observed in \cite{grave2019unsupervised}, the form of the approximate MLE $\hat{\Pi}_{\mathrm{AML}}$ in \prettyref{eq:MLEapprox1} as a double maximization over $\Pi \in \fS_n$ and $Q\in O(d)$ naturally suggests an alternating maximization strategy by iterating between the two steps:
(a) For a fixed $Q$, the $\Pi$-maximization is a linear assignment; 
(b) For a fixed $\Pi$, the $Q$-maximization is the so-called orthogonal Procrustes problem and easily solved via SVD \cite{schonemann1966generalized}. 
However, with random initialization this method performs rather poorly falling short of the optimal threshold predicted by \prettyref{thm:main}. While more  informative initialization (such as starting from a $\Pi$ obtained by the doubly-stochastic relaxation of QAP \cite{grave2019unsupervised}) can potentially help, in this section we focus on methods that are closer to the original approximate MLE.

As the proof of \prettyref{thm:main} shows, as far as achieving the optimal threshold is concerned it suffices to consider a finely discretized $O(d)$. 
This can be easily implemented in $d=2$, 
since any $2\times2$ orthogonal matrix is either a rotation or reflection of the form:
$\big(\begin{smallmatrix}
\cos(\theta) & -\sin(\theta)\\
\sin(\theta) & \cos(\theta)
\end{smallmatrix}\big)$
or 
$\big(\begin{smallmatrix}
\cos(\theta) & \sin(\theta)\\
\sin(\theta) & -\cos(\theta)
\end{smallmatrix}\big)$. 
We then solve
\prettyref{eq:MLEapprox1} on a grid of $\theta$ values, by 
solving the $\Pi$-maximization for each such $Q$ and reporting the solution with the highest objective value. As shown in \prettyref{fig:match_2D} for $n=200$, the performance of the approximate MLE in the dot-product model (green) follows closely that of the MLE in the linear assignment model (blue). Using the greedy matching algorithm (red) in place of the linear assignment solver greatly speeds up the computation at the price of some  performance degradation.

\begin{figure}[ht]
	\centering
	\subfigure[$n=200$ and $d=2$. The green and red curves correspond to  \prettyref{eq:MLEapprox1} on 
    $T_0=100$ discretized angles, with exact linear assignment or greedy matching.]%
	{\label{fig:match_2D} \includegraphics[width=0.48\columnwidth]{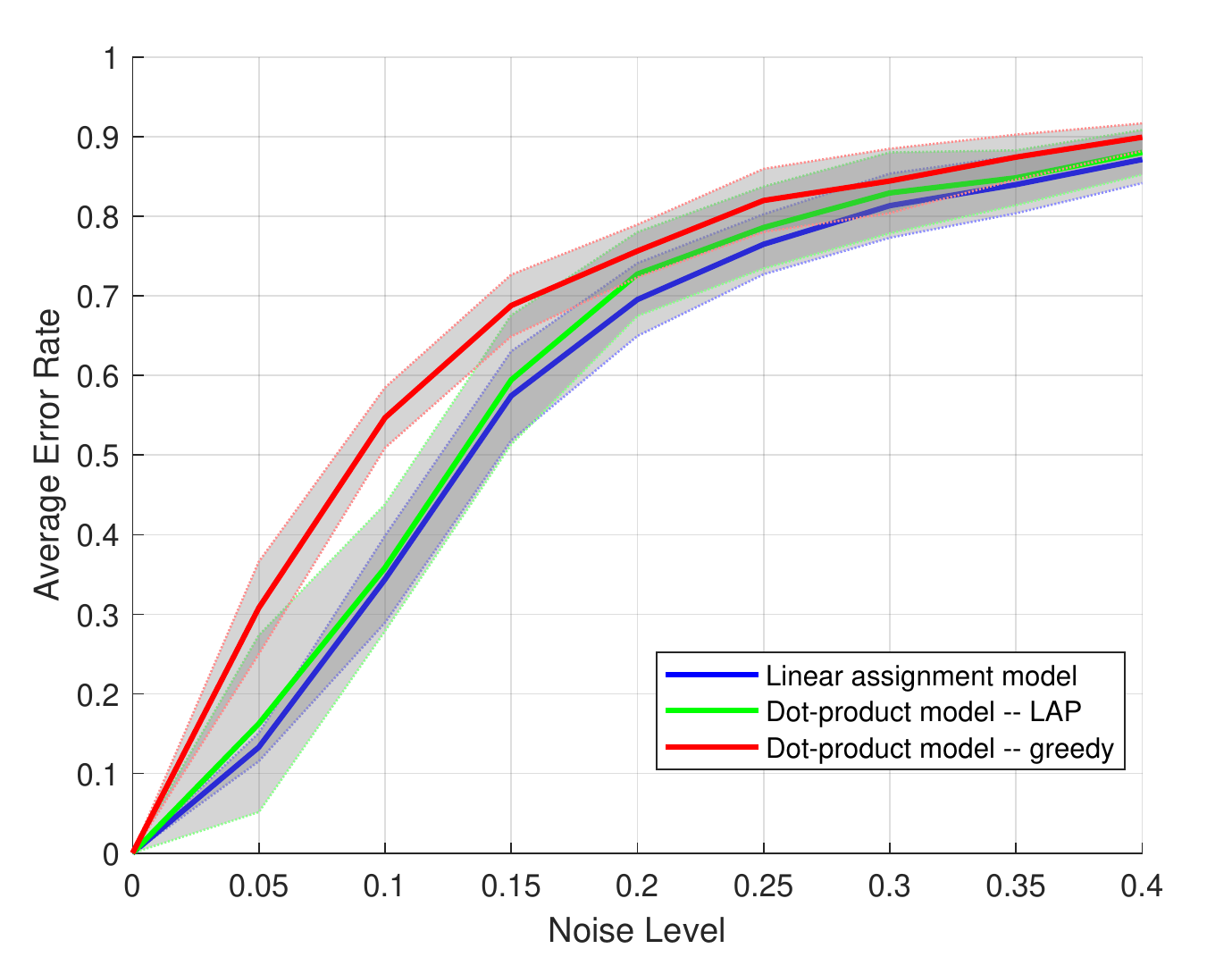}}
	\subfigure[$n=200$ and $d=4$. The green and red curves are based on \prettyref{eq:MLEapprox-Z2} with exact linear assignment or greedy matching, and the yellow curve corresponds to the Umeyama algorithm \prettyref{eq:umeyama}.]%
	{\label{fig:match_4D} \includegraphics[width=0.48\columnwidth]{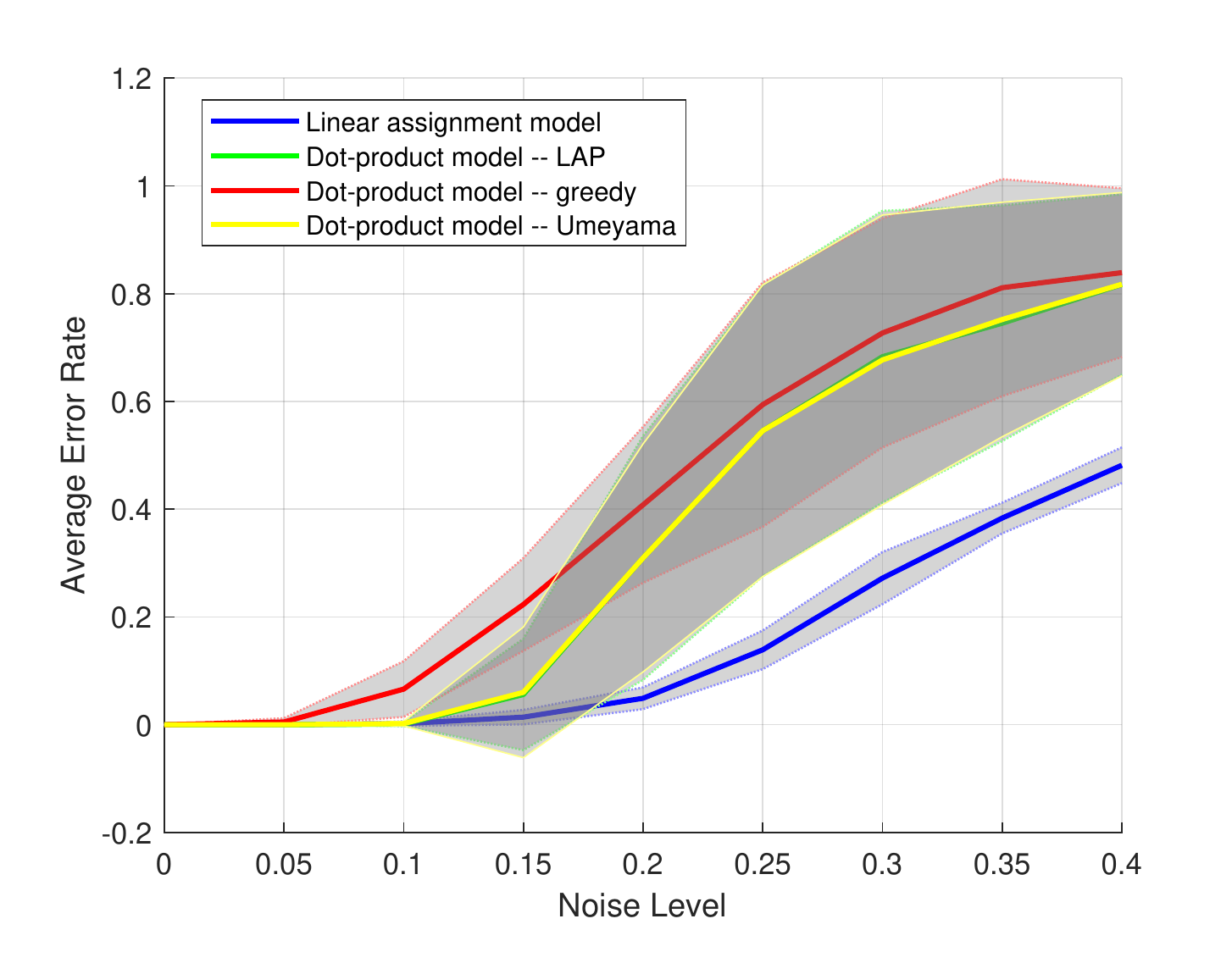}}
	\caption{Comparison of dot product model and linear assignment model (averaged over 10 random instances). In the latter, the blue curve corresponds to the MLE \prettyref{eq:LAP}.}
	\label{fig:match}
\end{figure}

As the dimension increases, it becomes more difficult and computationally more expensive to discretize $O(d)$. 
Instead, we take a different approach. Note that in the noiseless case ($Y=\Pi^* X$), as long as all singular values have multiplicity one, we have $B^{1/2} = \Pi^* A^{1/2} Q$ for some $Q$ in 
\begin{equation}
\integers_2^{\otimes d} = \{\diag(q_i): q_i \in \{\pm 1\}\}.    
\label{eq:Z2d}
\end{equation}
 As such, in the noiseless case it suffices to restrict the inner maximization of \prettyref{eq:MLEapprox1} to the subgroup $\integers_2^{\otimes d}$ corresponding to coordinate reflections. Since the noise is weak in the low-dimensional setting, we continue to apply this heuristic by computing
\begin{equation}
\hat{\Pi}_{\mathrm{AML}, \integers_2^{\otimes d}} = \arg\max_{\Pi \in \fS_n} \max_{Q\in \integers_2^{\otimes d}} \Iprod{B^{1/2}}{\Pi A^{1/2} Q},
\label{eq:MLEapprox-Z2}
\end{equation}
which turns out to work very well in practice. 
Taking this method one step further, notice that in the low-dimensional regime, all non-zero singular values of $X$ and $Y$ are tightly concentrated on the same value $\sqrt{n}$. If we ignore the singular values and simply replace $A^{1/2}$ and $B^{1/2}$ by their left singular vectors $U=[u_1,\ldots,u_d]$ and $V=[v_1,\ldots,v_d]$, \prettyref{eq:MLEapprox-Z2} can be written more explicitly as
\begin{equation}
\hat{\Pi}_{\mathrm{Umeyama}} = \arg\max_{\Pi \in \fS_n} \max_{q \in \{\pm1\}^d}
\iprod{\Pi}{\sum_{i=1}^d q_i  v_{i} u_i^\top},
\label{eq:umeyama}
\end{equation}
which, somewhat unexpectedly, coincides with the celebrated Umeyama algorithm~\cite{umeyama1988eigendecomposition}, a specific type of spectral method that is widely used in practice for graph matching.
In \prettyref{fig:match_4D} we compare 
for $n=200$ and $d=4$. 
Consistent with \prettyref{thm:main}, the error rates in the dot-product model and the linear assignment model are both near zero until $\sigma$ exceeds a certain threshold, after which the former departs from the latter. 
Finally, comparing \prettyref{fig:match_2D} and  \prettyref{fig:match_4D} confirms that 
the reconstruction threshold improves as the latent dimension increases as predicted by \prettyref{thm:main}.


\section{Discussion}
\label{sec:discuss}

In this paper we studied the problem of graph matching 
in the special case of correlated complete weighted graphs in the dot product and distance model, as  a first step towards the more challenging case of random dot-product graphs and random geometric graphs.
Within the confines of the present paper, there still remain a number of interesting directions and open problems which we discuss below.

\paragraph{Non-isotropic distribution}
The present paper assumes the latent coordinates $X_i$'s and $Y_i$'s are  isotropic Gaussians.
For the linear assignment model, \cite{dai2019database,dai2020achievability} has considered a more general setup where 
$X_i \iiddistr  N(0,\Sigma)$ for some covariance matrix $\Sigma$. 
As explained in~\cite[Appendix A]{dai2019database}, it is not hard to see, based on a simple reduction argument (by scaling both $X_i$'s and $Y_i$'s with $\Sigma^{1/2}$ and add noise if needed), that as long as the singular values of $\Sigma$ are bounded from above and below, the information-theoretic limits in terms of $\sigma$ remain unchanged.
For the dot product or distance model, this is also true but less obvious -- see \prettyref{app:nonisotropic} for a proof.

While the statistical limits in the nonisotropic case remain the same, potentially it allows more computationally tractable algorithms to succeed.
For example, the spectral method recently proposed in \cite{FMWX19a,FMWX19b} finds a matching by rounding  the so-called GRAMPA similarity matrix 
\begin{equation}
X = \sum_{i,j=1}^n \frac{\Iprod{u_i}{\ones}\Iprod{v_j}{\ones}}{(\lambda_i-\mu_j)^2+\eta^2} u_i v_j^\top.
\label{eq:GRAMPA}    
\end{equation}
Here $A=\sum \lambda_i u_iu_i^\top$ and 
$B=\sum \mu_j v_jv_j^\top$ are 
the SVD of the observed weighted adjacency matrices, and $\eta$ is a small regularization parameter.
In the isotropic case, applying this algorithm to the dot-product model is unlikely to achieve the optimal threshold in \prettyref{thm:main}. The reason is that in the low-dimensional regime of small $d$, both $A$ and $B$ and rank-$d$ and all  singular values $\lambda_i$'s and  $\mu_j $'s
are largely concentrated on the same value of $\sqrt{n}$. As such, the similarity matrix \prettyref{eq:GRAMPA} degenerates into 
$X \approx \frac{1}{n\eta^2} \sum_{i,j=1}^n
\lambda_i \mu_j \Iprod{u_i}{\ones}\Iprod{v_j}{\ones}  u_i v_j^\top \propto a b^\top$, where 
$a = A\ones$ and $b = B\ones$ are the row-sum vectors.
Rounding $ab^\top$ to a permutation matrix is equivalent to ``degree-matching'', that is, finding the permutation by sorting $a$ and $b$, which can only tolerate $\sigma = n^{-c}$ type of noise level, for constant $c$ \emph{independent of} the dimension $d$, due to the small spacing in the order statistics \cite{ding2018efficient}.
However, in the nonisotropic case where $\Sigma$ has distinct singular values, we expect $A$ and $B$ to have descent spectral gaps and the spectral method \prettyref{eq:GRAMPA} may succeed at the dimension-dependent thresholds of \prettyref{thm:main}. A theoretical justification of this heuristic is outside the scope of this paper.

\paragraph{High-dimensional regime}

Recall the exact MLE \prettyref{eq:MLE-dotprod}, wherein the objective function is an average over the Haar measure on $O(d)$, 
can be approximated by  \prettyref{eq:MLEapprox} for small $\sigma$.
Next, we derive its large-$\sigma$ approximation.
Rewriting the objective function in \prettyref{eq:MLE-dotprod}
as 
$\Expect[\exp(\frac{1}{\sigma^2} \Iprod{B^{1/2}}{\Pi A^{1/2} \mathbf{Q}})]$ 
for a random uniform $\mathbf{Q} \in O(d)$
and taking its second-order Taylor expansion for large $\sigma$, we get 
\[
\Expect\qth{\exp\pth{\frac{1}{\sigma^2} \Iprod{B^{1/2}}{\Pi A^{1/2} \mathbf{Q}}}}
= 1 + 
\frac{1}{2d\sigma^4} \Iprod{B}{\Pi A \Pi^\top}
+ o(\sigma^{-4}),
\]
where we applied 
$\Expect[\Iprod{\mathbf{Q}}{X}]=0$,  $\Expect[\Iprod{\mathbf{Q}}{X}^2]=\Fnorm{X}^2/d$, and $\fnorm{(A^{1/2})^\top \Pi^\top B^{1/2}}^2=\Iprod{B}{\Pi A \Pi^\top}$.
This expansion suggests that for large $\sigma$ (which can be afforded in the high-dimensional regime of $d \gg  \log n$), the MLE is approximated by the solution to the following QAP:
\begin{equation}
\hat{\Pi}_{\mathrm{QAP}} = \arg\max_{\Pi \in \fS_n} \Iprod{B}{\Pi A \Pi^\top}.
\label{eq:QAP}
\end{equation}
This observation aligns with the better studied correlated \ER models or correlated Gaussian Wigner models, where the MLE is exactly given by the QAP \prettyref{eq:QAP}.

To further compare with the estimator \prettyref{eq:MLEapprox} that has been shown optimal in low dimensions, 
let us rewrite \prettyref{eq:QAP} in a form that parallels \prettyref{eq:MLEapprox1}:
\begin{equation}
\hat{\Pi}_{\mathrm{QAP}} = 
\arg\max_{\Pi \in \fS_n} \fnorm{(A^{1/2})^\top \Pi^\top B^{1/2}}
=
\arg\max_{\Pi \in \fS_n} \max_{\fnorm{Q}\leq 1} \Iprod{B^{1/2}}{\Pi A^{1/2} Q}.
\label{eq:MLEapprox3}
\end{equation}
In contrast, the dual variable $Q$ in  \prettyref{eq:MLEapprox1} is constrained to be an orthogonal matrix, which, as discussed in the proof sketch in \prettyref{sec:positve-sketch}, is crucial for the proof of \prettyref{thm:main}.
Overall, the above evidence points to the potential suboptimality of QAP in low and moderate-dimensional regime of $d \lesssim \log n$ and its potential optimality in the high-dimensional regime of $d \gg \log n$.

\paragraph{Practical algorithms}

As demonstrated by extensive numerical experiments in~\cite[Sec.~4.2]{FMWX19a}, for correlated random graph models with iid pairs of edge weights, the Umeyama algorithm \prettyref{eq:umeyama} significantly improves over 
classical ``low-rank'' spectral methods involving only the top few eigenvectors, but still lags behind the more recent spectral methods such as the GRAMPA algorithm~\prettyref{eq:GRAMPA} that uses all pairs of eigenvalues and eigenvectors. Surprisingly, in the low-dimensional dot product model with $d=o(\log n)$, while the GRAMPA algorithm is expected to perform poorly, empirical result in \prettyref{sec:exp} indicates that the Umeyama method actually works very well in this setting. 
In fact, it is not hard to show
that the  Umeyama algorithm returns the true permutation
with high probability  in the noiseless case of $\sigma=0$;
however, understanding its theoretical performance in the noisy setting remains open. 

\appendix
\section{Further related work}
	\label{app:related}
The present paper bridges several streams of literature such as planted matching,  feature matching, Procrustes matching, and graph matching, which we describe below.

 \paragraph*{Planted matching and feature matching} The planted matching problem aims to recover a perfect matching hidden in a weighted complete $n\times n$ bipartite graph, where the edge weights are independently drawn from either  $\calP$ or $\calQ$ depending on whether edges are on the hidden matching or not. Originally proposed by~\cite{Chertkov2010} to model the application of object tracking, a sharp phase transition from almost perfect recovery to partial recovery is conjectured to exist for the special case where $\calP$ is a folded Gaussian and $\calQ$ is a uniform distribution over $[0,n]$.
A recent line of work initiated by~\cite{moharrami2021planted} and followed by~\cite{semerjian2020recovery,DWXY21} has successfully resolved the conjecture and characterized the sharp threshold for general distributions. 

Despite these fascinating advances, they crucially rely on the independent weight assumption which does not account for the latent geometry in the object tracking applications. As a remedy, the linear assignment model~\prettyref{eq:model-LAP} was proposed and studied by~\cite{kunisky2022strong} as a geometric model for planted matching, where the edge weights are pairwise inner products and no longer independent. 
In the low-dimensional setting of $d=o(\log n)$, the MLE is shown to achieve perfect recovery when $\sigma=o(n^{-2/d})$ and almost perfect recovery when $\sigma=o(n^{-1/d})$. Further bounds on the number of errors made by MLE and recovery guarantees in the high-dimensional setting are provided. 
However, the necessary conditions derived in \cite{kunisky2022strong} only pertain to the MLE, leaving open the possibility that almost perfect recovery might be attained by other algorithms at lower threshold.
This is resolved in the negative by the information-theoretic converse in \prettyref{thm:opt}, showing that $\sigma=o(n^{-1/d})$ is necessary for any algorithm to achieve almost perfect recovery. 
Along the way, we also slightly improve the necessary condition for perfect recovery from $\sigma=O(n^{-2/d})$ to $\sigma=o(n^{-2/d})$.


The  linear assignment model~\prettyref{eq:model-LAP} was in fact studied earlier in \cite{dai2019database,dai2020achievability} in a different context of feature matching, where $X_i$'s and $Y_i$'s are viewed as two correlated Gaussian feature vectors in $\reals^d$ and the goal is to find their best alignment.
It is shown in~\cite{dai2019database} that perfect recovery is possible when $\frac{d}{4} \log \left(1+\sigma^{-2} \right) - \log n \to +\infty$, and impossible 
when $\frac{d}{4} \log \left(1+\sigma^{-2} \right) \le (1-\Omega(1))\log n$ and $1 \ll d =O(\log n)$.\footnote{While the impossibility result in \cite[Theorem 2]{dai2019database}  only states the assumption that $d \gg 1$, its proof,
specifically the proof of \cite[Lemma 4.5]{dai2019database}, implicitly assumes $\sigma=O(1)$ which further implies $d=O(\log n)$.} In comparison, the necessary condition in~\prettyref{thm:exact_nec} is tighter and 
holds for any $d$, agreeing with their sufficient condition within an additive $\log d$ factor. 
It is also shown in~\cite{dai2020achievability} that almost perfect recovery is possible when $\frac{d}{2} \log \left(1+\sigma^{-2} \right) \ge (1+\epsilon) \log n$ in the high-dimensional regime
$d=\omega(\log n)$ for a small constant $\epsilon>0$. This matches our necessary condition in~\prettyref{prop:impossiblity} with a sharp constant.



Related problems on feature matching were also studied in the statistics literature. For example, \cite{collier2016minimax} studies the model of observing $Y=X+\sigma Z$ and $Y'=\Pi^* X+ \sigma Z'$, where 
$Z,Z'$ are two independent random Gaussian matrices
and $X$ is deterministic.  The minimum separation (in Euclidean distance) of rows of $X$ needed for perfect recovery, denoted by $\kappa$, is shown to be on the order of  $\sigma \max\{ (\log n)^{1/2}, (d\log n)^{1/4} \}$. Note that in the low-dimensional regime $d=o(\log n)$, this condition is comparable to our threshold for perfect recovery $\sigma=o(n^{-2/d})$, as the typical value of $\kappa$ scales as $n^{-2/d}$ when $X$ is Gaussian. However, the average-case setup is more challenging as $\kappa$ can be atypically small due to the stochastic variation of $X$.

\paragraph*{Procrustes matching}	
Our dot-product model is also closely related to the problem of Procrustes matching, which finds numerous applications in natural language processing
and computer vision~\cite{rangarajan1997softassign,maron2016point,dym2017exact,grave2019unsupervised}. Given two point clouds stacked as rows of $X$ and $Y$, Procrustes matching aims to find an orthogonal matrix $Q \in O(d)$
and a permutation $\Pi \in \fS_n$ that minimizes the Euclidean distance between the point
clouds, \ie, $\min_{\Pi \in \fS_n} \min_{Q \in O(d)} \fnorm{Y Q-\Pi X }^2$.
As observed in~\cite{grave2019unsupervised}, this is equivalent to $\max_{\Pi \in \fS_n} \max_{Q \in O(d)} \iprod{Y Q}{\Pi X}$,
which further reduces to $\max_{\Pi \in \fS_n} \|X^\top \Pi^\top Y\|_\ast$. Thus our approximate MLE 
\prettyref{eq:MLEapprox}
under the dot-product model is equivalent to Procrustes matching on $A^{1/2}$ and $B^{1/2}$. 
A semi-definite programming relaxation is proposed in~\cite{maron2016point} and further shown to return the optimal solution 
in the noiseless case 
when $X$ is generic and asymmetric~\cite{maron2016point,dym2017exact}. In contrast, the more recent work \cite{grave2019unsupervised} proposes an iterative algorithm  based on the alternating maximization over $\Pi$ and $Q$ with an initialization provided by solving a doubly-stochastic relaxation of the QAP $\max_{\Pi \in \fS_n} \fnorm{X^\top \Pi^\top Y}^2$. Its performance is empirically evaluated on real datasets, but no theoretical performance guarantee is provided. Since the dot-product model is equivalent to the statistical model for Procrustes matching, where $Y=\Pi^* X Q+ \sigma Z$ for a random permutation $\Pi^*$ and orthogonal matrix $Q$, our results in \prettyref{thm:main} and \prettyref{thm:opt} thus characterize the statistical limits of Procrustes matching.


 \paragraph*{Graph matching}
 There has been a recent surge of interest in understanding the information-theoretic and algorithmic limits of random graph matching~\cite{cullina2016improved,cullina2017exact,Hall2020partial,wu2021settling,ding2018efficient,barak2019nearly,FMWX19a,FMWX19b,ganassali2020tree,ganassali2021correlation,mao2021random,mao2021exact}, which 
is an average-case model for the QAP and a noisy version of random graph isomorphism \cite{babai1980random}. Most of the existing work is restricted to the correlated \ER-type models in which $\left( A_{\pi^*(i)\pi^*(j)}, B_{ij}\right)$ are iid pairs of two correlated Bernoulli or Gaussian random variables. In this case, the maximum likelihood estimator reduces to solving the QAP \prettyref{eq:QAP}.
Sharp information-theoretic limits are derived by analyzing this QAP~\cite{cullina2016improved,cullina2017exact,ganassali2021sharp,wu2021settling}
 and various efficient algorithms are developed based on its spectral or convex relaxations~\cite{umeyama1988eigendecomposition,zaslavskiy2008path,aflalo2015convex,vogelstein2015fast,lyzinski2016graph,dym2017ds++,FMWX19a,FMWX19b}.
 However, as discussed in \prettyref{sec:discuss}, for geometric models such as the dot-product model, the QAP 
 is the high-noise approximation of the MLE \prettyref{eq:MLE-dotprod}, which differs from the low-noise approximation~\prettyref{eq:MLE-dotprod} that is shown to be optimal in the low-dimensional regime of $d=o(\log n)$. 
 This observation suggests that for geometric models one may need to  
 rethink the algorithm design and move beyond the QAP-inspired methods.


\section{Maximal likelihood estimator in the dot-product model}
\label{app:MLE}
To compute the ``likelihood'' of the observation $(A,B)$ given the ground truth $\Pi^*$, it is useful to keep in mind of the graphical model
\[
\begin{tikzcd}
\Pi^* \arrow[r] & Y \arrow[r]  & B \\
& X \arrow[u] \arrow[r] & A
\end{tikzcd}
\]
where $X,Y,\Pi^*$ is related via \prettyref{eq:model-LAP}, $A=XX^\top$, and $B=YY^\top$.




Note that $A$ are $B$ are rank-deficient. To compute the density of $(A,B)$ conditioned on $\Pi^*$ meaningfully, one needs to choose an appropriate reference measure $\mu$ and evaluate the relative density $\frac{\diff P_{A,B|\Pi^*}}{\diff \mu}$.
Let us choose $\mu$ to be the product of the marginal distributions of $A$ and $B$, which does not depend on $\Pi^*$. 
For any rank-$d$ positive semidefinite matrices $A_0$ and $B_0$,
define $A_0^{1/2} \triangleq U_0 \Lambda^{1/2}$ and $B_0^{1/2} \triangleq V_0 D^{1/2}$ based on the SVD
$A_0 = U_0 \Lambda_0^{1/2} Q_0^\top$ and $B_0 = V_0 D_0 O_0^\top$, 
where $Q_0,O_0\in O(d)$ and $U_0,V_0 \in V_{n,d} \triangleq \{U\in\reals^{n\times d}: U^\top U = I_d\}$ (the Stiefel manifold). We aim to show 
\begin{equation}
\frac{\diff P_{A,B|\Pi^*}(A_0,B_0|\Pi)}{\diff \mu(A_0,B_0)} = h(A_0,B_0) \int_{O(d)} \diff Q \exp\pth{\frac{\Iprod{B_0^{1/2}}{\Pi A_0^{1/2} Q}}{\sigma^2}}
\label{eq:likelihood-dotprod}
\end{equation}
for some fixed function $h$, where the integral is with respect to the Haar measure on $O(d)$. 
This justifies the MLE in \prettyref{eq:MLE-dotprod} for the dot-product model.

To show \prettyref{eq:likelihood-dotprod}, denote by $N_\delta(U_0) = \{U \in V_{n,d}: \fnorm{U-U_0} \leq \delta\}$  and $N_\delta(\Lambda_0) = \{\Lambda \text{ diagonal}: \|\Lambda-\Lambda_0\|_{\ell_\infty}\leq \delta\}$  neighborhoods of $U_0$ and $\Lambda_0$ respectively. (Their specific definitions are not crucial.)
Consider a $\delta$-neighborhood of $A_0$ of the following form:
\[
N_\delta(A_0) \triangleq \{U \Lambda U^\top: U \in N_\delta(U_0), \Lambda \in N_\delta(\Lambda_0)\}
\]
and similarly define $N_\delta(B_0)$.
Write the SVD for $X$ as $X= U R Q^\top$, where $U\in V_{n,d}, Q\in O(d)$ and the diagonal matrix $R$ are mutually independent; in particular, $Q$ is uniformly distributed over $O(d)$.
Then
for constant $C=C(n,d,\sigma)$,
\begin{align*}
&~\Prob[A \in N_\delta(A_0),B \in N_\delta(B_0) | \Pi^*=\Pi] \\
= & ~ \Expect[\indc{XX^\top \in N_\delta(A_0)} \indc{YY^\top \in N_\delta(B_0)}  | \Pi^*=\Pi] \\
= & ~ \Expect[\indc{U \in N_\delta(U_0)} \indc{R \in N_\delta(D_0^{1/2})}  \indc{YY^\top \in N_\delta(B_0)}  | \Pi^*=\Pi] \\
= & ~ C\cdot \Expect\qth{\indc{U \in N_\delta(U_0)} \indc{R \in N_\delta(D_0^{1/2})}  \int_{\reals^{n\times d}} \diff y \indc{yy^\top \in N_\delta(B_0)} \exp\pth{-\frac{\fnorm{y-\Pi U RQ^\top}^2}{2\sigma^2}} } \\
= & ~ C \cdot \Expect\qth{\indc{U \in N_\delta(U_0)} \indc{R \in N_\delta(D_0^{1/2})} \int_{\reals^{n\times d}} \diff y \indc{yy^\top \in N_\delta(B_0)} 
\exp\pth{-\frac{\fnorm{y}^2+ \fnorm{R}^2}{2\sigma^2}} F(y,\Pi U R)},
\end{align*}
where $F:\reals^{n\times d}\times \reals^{n\times d}\to\reals_+$ is defined by
\[
F(y,x)\triangleq \Expect_Q\qth{\exp\pth{\frac{\Iprod{y}{x Q^\top}}{\sigma^2}}}  = \int_{O(d)} \diff Q \exp\pth{\frac{\Iprod{y}{xQ^\top}}{\sigma^2}}.
\]
Note that this function is continuous, strictly positive, and right-invariant, in the sense that $F(YO,XO')=F(Y,X)$ for any $O,O'\in O(d)$. Thus, as $\delta \to 0$, we have
for some constant $C'=C'(n,d,\sigma)$,
\begin{align*}
&~\Prob[A \in N_\delta(A_0),B \in N_\delta(B_0) | \Pi^*=\Pi] \\
= & ~ (1+o(1)) \underbrace{C'  \exp\pth{\frac{\Tr(A_0)}{2\sigma^2} - \frac{\Tr(B_0)}{2\sigma^2(\sigma^2+1)}  }}_{\triangleq h(A_0,B_0)} F(B_0^{1/2},\Pi A_0^{1/2}) \\
& \cdot \underbrace{\Expect\qth{\indc{U \in N_\delta(U_0)} \indc{R \in N_\delta(D_0^{1/2})} } \cdot (2\pi (1+\sigma^2))^{-nd/2}
 \int_{\reals^{n\times d}} \diff y \indc{yy^\top \in N_\delta(B_0)}  \exp\pth{-\frac{\fnorm{y}^2}{2(1+\sigma^2)}}}_{\mu[A \in N_\delta(A_0),B \in N_\delta(B_0)]},
\end{align*}
proving \prettyref{eq:likelihood-dotprod}.

\section{Analysis of approximate maximum likelihood}
\label{sec:positve}

In this section we prove 
Theorem \ref{thm:main} for the dot product model. The proof of \prettyref{thm:distance} for the distance model follows the same program and is postponed to \prettyref{app:distance}.

\subsection{Discretization of orthogonal group}


We first prove \prettyref{lem:Net_Error} on 
the approximation of nuclear norm on a discretization of $O(d)$.

\begin{proof}[Proof of Lemma \ref{lem:Net_Error}]
Consider the singular value decomposition $ A=UDV^\top $, where $U,V\in O(d)$ and $D$ is diagonal. Then the nuclear norm $ \|A\|_*=\max_{Q \in O(d)} \iprod{A}{Q} = \Tr(D)$ is attained at $ Q_*=UV^\top $. Pick an element $ Q \in N $ with $ Q=Q_*+\Delta $, where $ \|\Delta\| \leq \delta $. By orthogonality of $Q$ and $Q_*$, we have
\begin{equation}
\Delta Q_*^\top + Q_* \Delta^\top + \Delta \Delta^\top=0.
\label{eq:Q*_Delta}
\end{equation}
Note that
\begin{equation}
A Q_*^\top = Q_* A^\top = U D U^\top =:B.
\label{eq:A_Q*}
\end{equation}
Also, we have
$$ \iprod{A}{\Delta} = \iprod{A Q_*^\top}{\Delta Q_*^\top},\ \ \ \iprod{A}{\Delta}=\iprod{A^\top}{\Delta^\top}=\iprod{Q_* A^\top}{Q_* \Delta^\top}. $$
Adding the above equations and applying \prettyref{eq:Q*_Delta}-\prettyref{eq:A_Q*} yield
$$ \iprod{A}{\Delta}=\frac{1}{2} \iprod{B}{\Delta Q_*^\top + Q_* \Delta^\top} = -\frac{1}{2} \iprod{B}{\Delta \Delta^\top}. $$
This implies
$$ \left| \iprod{A}{\Delta} \right| \leq \frac{1}{2} \|B\|_* \|\Delta\|^2 = \frac{1}{2} \|A\|_* \|\Delta\|^2,  $$
which completes the proof.
\end{proof}

Next we give a specific construction of a $ \delta $-net for $O(d)$ that is suitable for the purpose of proving \prettyref{thm:main}. 
Since orthogomal matrices are normal, by the spectral decomposition theorem, each orthogonal matrix $ Q \in O(d) $ can be written as $ Q=U^* \Lambda U $, where $ \Lambda=\diag(e^{\ii \theta_1},\dots,e^{\ii \theta_d}) $ with $ \theta_j \in [-\pi,\pi] $ for all $ j=1,\dots,d $ and $ U \in U(d) $ is an unitary matrix.
To construct a net for $O(d)$, we first discretize the eigenvalues uniformly and then discretize the eigenvectors according to the optimal local entropy of orthogonal matrices with prescribed eigenvalues.

For any fixed $ \delta>0 $, let $ \Theta \triangleq \{\theta_k=\tfrac{k \delta}{4}:k=\lfloor -\tfrac{4 \pi}{\delta} \rfloor, \lfloor -\tfrac{4 \pi}{\delta} \rfloor +1 ,\dots,\lceil \tfrac{4 \pi}{\delta} \rceil\} $. Then the set 
$$ \mathbf{\Lambda} \triangleq \left\{ (\lambda_1,\dots,\lambda_d) \in \C^d: \lambda_j=e^{\ii \theta_j}, \theta_j \in \Theta, j=1,\dots,d \right\} $$ 
is a $ \tfrac{\delta}{4} $-net in $ \ell_\infty $ norm for the set of all possible spectrum $ \{(\lambda_1,\dots,\lambda_d) \in \C^d:|\lambda_j|=1\} $. For each $ (\lambda_1,\dots,\lambda_d) \in \C^d $, let $ O(\lambda_1,\dots,\lambda_d) $ denote the set of orthogonal matrices with a prescribed spectrum $ \{\lambda_j\}_{j=1}^d $, i.e.
$$ O(\lambda_1,\dots,\lambda_d) \triangleq \left\{ O \in O(d) : \lambda_i(O)=\lambda_i,i=1,\dots,d \right\}, $$
where $ \lambda_i(O) $'s are the eigenvalues of $ O $ sorted in the counterclockwise way from $ -\pi $ to $ \pi $.
Similarly, define $ U(\lambda_1,\dots,\lambda_d) $ to be the set of unitary matrices with a given spectrum
$$ U(\lambda_1,\dots,\lambda_d) \triangleq \{U^* \diag(\lambda_1,\ldots,\lambda_d) U: U\in U(d)\}. $$
Then $ O(\lambda_1,\dots,\lambda_d) \subset U(\lambda_1,\dots,\lambda_d) \subset U(d) $. Let $ N'(\lambda_1,\dots,\lambda_d) $ be the optimal $ \tfrac{\delta}{4} $-net in operator norm for $ U(\lambda_1,\dots,\lambda_d) $, and let $ N(\lambda_1,\dots,\lambda_d) $ be the projection (with respect to $\Opnorm{\cdot}$) of $ N'(\lambda_1,\dots,\lambda_d) $
to $O(d)$. Define
\begin{equation}
N \triangleq \bigcup_{(\lambda_1,\dots,\lambda_d) \in \mathbf{\Lambda}} N(\lambda_1,\dots,\lambda_d). 
\label{eq:net}
\end{equation}
We claim that $ N $ is a $ \delta $-net in operator norm for the orthogonal group.

\begin{lemma}
The set $ N \subset O(d) $ defined in \prettyref{eq:net} is a $ \delta $-net in operator norm for $ O(d) $.
\end{lemma}

\begin{proof}
Given $Q\in O(d)$, let its eigenvalue decomposition be $Q=U^* \Lambda U$. 
where $\Lambda=\diag(\lambda_1,\ldots,\lambda_d)$. Then there exists 
$\tilde \Lambda=\diag(\tilde \lambda_1,\ldots,\tilde \lambda_d)$ where $(\tilde \lambda_1,\ldots,\tilde \lambda_d) \in \mathbf{\Lambda}$, such that
$\|\Lambda-\tilde \Lambda\|\leq \tfrac{\delta}{4}$.
By definition, there exists $ \tilde U \in U(d)$ such that 
$\tilde U^*\tilde \Lambda \tilde U\in N'(\tilde \lambda_1,\ldots,\tilde \lambda_d)$ and 
$\|\tilde U^*\tilde \Lambda \tilde U- U^*\tilde \Lambda  U\| \leq \tfrac{\delta}{4}$. 
Let $\tilde Q \in N$ denote the projection of $\tilde U^*\tilde \Lambda \tilde U$.
Then
\begin{align*}
\|Q-\tilde Q\| \leq & ~ \| Q -\tilde U^*\tilde \Lambda \tilde U\|+ \| \tilde U^*\tilde \Lambda \tilde U -\tilde Q\|\\
\leq & ~ 2\|\tilde U^*\tilde \Lambda \tilde U- Q\| \\
= & ~ 2\|\tilde U^*\tilde \Lambda \tilde U- U^*\Lambda U\| \\
\leq & ~ 2(\|\tilde U^*\tilde \Lambda \tilde U- U^*\tilde \Lambda U\| + \|U^*(\tilde \Lambda -\Lambda)  U\|) \leq \delta,
\end{align*}
where the second inequality follows from projection.
\end{proof}

The size of this $ \delta $-net is estimated in the following lemma.

\begin{lemma}[Local entropy of $ O(d) $]\label{lem:Net_Size}
For each $ (\lambda_1,\dots,\lambda_d) $ where $ \lambda_\ell=e^{\ii \theta_\ell} $, we have
\begin{equation}
|N(\lambda_1,\dots,\lambda_d)| \leq \pth{1+\frac{2 \max |\theta_\ell|}{\delta}}^{2d^2}
\end{equation}
\end{lemma}

\begin{proof}
Note that
\begin{equation*}
U(\lambda_1,\dots,\lambda_d) =I+\sth{ U^* \diag \pth{\lambda_1-1,\dots,\lambda_d-1}U:U \in U(d) }
=: I+\tilde{U}(\lambda_1,\dots,\lambda_d).
\end{equation*}
For any matrix $ Q \in \tilde{U}(\lambda_1,\dots,\lambda_d) $, we have
$$ \opnorm{Q}^2 = \max \left| e^{\ii \theta_\ell} -1 \right|^2 = \max |2-2\cos \theta_\ell| \leq \max |\theta_\ell|^2. $$
where $ \Opnorm{\cdot} $ is the the operator norm with respect to $ \C^d \to \C^d $.
This implies
$$ U(\lambda_1,\dots,\lambda_d) \subset \mathrm{B}(I,\max |\theta_\ell|), $$
where $  \mathrm{B}(I,r) $ is the operator norm ball centered at $ I_d $ with radius $ r $.  As a normed vector space over $ \R $, the space of $ d \times d $ complex matrices has dimension $ 2d^2 $ since $ \C^{d \times d} \simeq \R^{2d^2} $. Then the desired result follows from a standard volume bound
(c.f.~e.g.~\cite[Lemma 4.10]{Pisier-book}) 
for the metric entropy
$$ \left| N(\lambda_1,\dots,\lambda_d) \right| \leq \left| N'(\lambda_1,\dots,\lambda_d) \right| \leq \pth{1+\frac{2 \max |\theta_\ell|}{\delta}}^{2d^2}. $$
\end{proof}

\subsection{Moment generating functions and cycle decomposition}

Based on the reduction \prettyref{eq:Prob_Reduction}, it suffices to estimate
$$ \sum_{\Pi \neq I_n} \sum_{(\lambda_1,\dots,\lambda_d) \in \mathbf{\Lambda}} \sum_{Q \in N(\lambda_1,\dots,\lambda_d)} p(\Pi,Q), $$
where
\begin{equation}
p(\Pi,Q) \triangleq \expect \exp \left\{ -\frac{1}{32 \sigma^2} \Fnorm{X-\Pi X Q}^2 \right\}.
\label{eq:MGF}
\end{equation}
This moment generating function (MGF) is estimated in the following lemma.

\begin{lemma}\label{lem:MGF}
For any fixed $ \Pi \in \fS_n $, let $ \calO $ denote the set of orbits of the permutation and $ n_k $ be the number of orbits with length $ k $. 
Let $Q \in O(d)$ and denote by $ e^{\ii \theta_1},\dots,e^{\ii \theta_d} $ the eigenvalues of $ Q $, where $ \theta_1,\dots,\theta_d \in [-\pi,\pi] $.
Then
\begin{equation}
p(\Pi,Q) = \prod_{O \in \calO,|O| \geq 1} a_{|O|}(Q) = \prod_{k=1}^n a_k(Q)^{n_k}, \label{eq:MGF_Compute}
\end{equation}
where
\begin{equation}
a_k(Q) \triangleq (4\sigma)^{kd} \prod_{\ell=1}^d \qth{ (\sqrt{1+4\sigma^2}+2\sigma)^{2k}+(\sqrt{1+4\sigma^2}-2\sigma)^{2k} - 2 \cos(k \theta_\ell)  }^{-1/2}, \label{eq:ak}
\end{equation}
satisfying, for all $ 1 \leq k \leq n $,
\begin{equation}
a_k(Q) \leq a_k(I) \leq (4 \sigma)^{(k-1)d}. \label{eq:ak_Estimate}
\end{equation}
Furthermore,
\begin{equation}
a_1(Q) \leq (C\sigma)^d \prod_{\ell=1}^d \frac{1}{\sigma+|\theta_\ell|}, \label{eq:a1_Estimate}
\end{equation}
where $ C>0 $ is a universal constant independent of $ d,n,\sigma $.
\end{lemma}

\begin{proof}
For simplicity, denote $ t=\tfrac{1}{32 \sigma^2} $. Let $ x= \vecc(X) \in \R^{nd} $ be the vectorization of $ X $, and note that $ x \sim \calN(0,I_{nd}) $. Through the vectorization, we have
$$ \Fnorm{X-\Pi X Q}^2 = \norm{(I_{nd} - Q^\top \otimes \Pi)x}^2. $$
Let $ H \triangleq I_{nd} - Q^\top \otimes \Pi $, then
\begin{equation}
p(\Pi,Q) = \expect \exp \pth{ -t x^\top H^\top H x } = \qth{ \det \pth{I+2tH^\top H} }^{-\frac{1}{2}}.
\label{eq:MGF_Determinant}
\end{equation}
Note that the eigenvalues of $ H $ are
$$ \lambda_{ij}(H)=1-\lambda_i(Q^\top)\lambda_j(\Pi),\ \ i=1,\dots,d,\ j=1,\dots,n. $$
This leads to
\begin{equation}
p(\Pi,Q) = \prod_{i=1}^d \prod_{j=1}^n \pth{1+2t \left|1-\lambda_i(Q^\top)\lambda_j(\Pi) \right|^2}^{-\frac{1}{2}}. \label{eq:MGF_Expansion}
\end{equation}

Through a cycle decomposition, the spectrum of $ \Pi $ is the same as a block diagonal matrix $ \tilde{\Pi} $ of the following form
$$ \tilde{\Pi} = \diag \pth{P_1^{(1)},\dots,P_{n_1}^{(1)},\dots,P_{1}^{(k)},\dots,P_{n_k}^{(k)},\dots,P_{1}^{(n)},\dots,P_{n_n}^{(n)}}, $$
where $ n_k $ is the number of $ k $-cycles in $ \pi $, and $ P_1^{(k)}=\cdots=P_{n_k}^{(k)} = P^{(k)} $ is a $ k \times k $ circulant matrix given by
\begin{equation*}
P^{(k)}=
\begin{bmatrix}
0 & 1 & \cdots & & 0\\
0 & 0 & 1 & & &\\
\vdots & 0 &0& \ddots & \vdots &\\
& & \ddots & \ddots & 1\\
1 & & \cdots &0 & 0
\end{bmatrix}.
\end{equation*}
It is well known that the eigenvalues of $ P^{(k)} $ are the $ k $-th roots of unity $ \{e^{\ii \frac{2\pi}{k}j}\}_{j=0}^{k-1} $. 
Therefore, the spectrum of $ \Pi $ is the following multiset
\begin{equation}
\mathsf{Spec}(\Pi)=\{e^{\ii \frac{2\pi}{k}j_k} \mbox{ with multiplicity } n_k : 1 \leq k \leq n,j_k=0,\dots,k-1\}. \label{eq:Spectrum_Pi}
\end{equation}
Recall that $ e^{\ii\theta_1},\dots,e^{\ii\theta_d} $ are the eigenvalues of $ Q $. Note that the eigenvalues of $ Q^\top $ are the complex conjugate of the eigenvalues of $ Q $. Combined with \prettyref{eq:MGF_Expansion} and \prettyref{eq:Spectrum_Pi}, we have
\begin{align}
p(\Pi,Q) &= \qth{ \prod_{\ell=1}^d \prod_{k=1}^n \prod_{j=0}^{k-1} \pth{1+2t \left| 1-e^{-\ii \theta_\ell}e^{\ii \frac{2\pi}{k}j}  \right|^2}^{n_k} }^{-1/2} \nonumber \\
&= \prod_{k=1}^n \qth{ \prod_{\ell=1}^d \prod_{j=0}^{k-1} \pth{ 1+4t-4t \cos(-\theta_\ell + \tfrac{2\pi}{k}j) }^{-1/2} }^{n_k} \triangleq \prod_{k=1}^n a_k(Q)^{n_k}. \label{eq:MGF_ak}
\end{align}
Define
$$ f(\theta) \triangleq \prod_{j=0}^{k-1} \pth{ 1+4t-4t \cos(\theta + \tfrac{2\pi}{k}j)}, $$
To simplify $ f(\theta) $, let $ p=\tfrac{\sqrt{1+8t}+1}{2} $ and $ q=\tfrac{\sqrt{1+8t}-1}{2} $ so that $ p^2+q^2=1+4t $ and $ pq=2t $. Thus,
$$ f(\theta) = \prod_{j=0}^{k-1} \pth{ p^2+q^2 -2pq \cos \pth{\tfrac{2\pi}{k}j + \theta} }. $$
Note that
$$ p^k - q^k e^{\ii k\theta}= \prod_{j=0}^{k-1} \pth{p-q e^{\ii \frac{2\pi}{k}j + \ii \theta}}, \ \ p^k- q^k e^{-\ii k\theta}= \prod_{j=0}^{k-1} \pth{p-q e^{\ii \frac{2\pi}{k}j - \ii \theta}}. $$
Multiplying the above two equations gives us
$$ p^{2k}+q^{2k}-2p^k q^k \cos k \theta = \prod_{j=0}^{k-1} \pth{ p^2+q^2 -2pq \cos \pth{\tfrac{2\pi}{k}j + \theta} } = f(\theta). $$
which implies
\begin{align*}
f(\theta) &= \pth{\frac{\sqrt{1+8t}+1}{2}}^{2k} + \pth{\frac{\sqrt{1+8t}-1}{2}}^{2k} - 2 (2t)^k \cos(k\theta)\\
&=\pth{\frac{1}{4\sigma}}^{2k} \qth{ \pth{\sqrt{1+4\sigma^2} + 2\sigma}^{2k} + \pth{\sqrt{1+4\sigma^2} - 2\sigma}^{2k} - 2\cos k\theta }.
\end{align*}
Note that $ a_k(Q) = \prod_{\ell=1}^d f(-\theta_\ell)^{-1/2} $, and therefore we have shown \prettyref{eq:ak}. In particular,
\begin{equation}
a_1(Q) = (4\sigma)^d \prod_{\ell=1}^d (2-2\cos \theta_\ell + 16 \sigma^2)^{-\frac{1}{2}}. \label{eq:a1}
\end{equation}
Since $ \sin^2 \theta \geq \tfrac{\theta^2}{4}  $ for $ \theta \in [-\tfrac{\pi}{2},\tfrac{\pi}{2}] $, we have
\begin{multline*}
\sqrt{2-2\cos \theta_\ell + 16 \sigma^2} = \sqrt{4\sin^2(\theta_\ell/2) + 16 \sigma^2}
\ge \sqrt{2\sin^2(\theta_\ell/2) }+ \sqrt{8\sigma^2}\\
\ge\sqrt{ 2 (\theta_\ell/4)^2} + \sqrt{8 \sigma^2} = \sqrt{2} |\theta_\ell|/4 + 2 \sqrt{2} \sigma
\end{multline*}
Consequently, this gives us \prettyref{eq:a1_Estimate}. 
In general, note that
$$
(\sqrt{1+4\sigma^2} + 2\sigma)^{2k} + (\sqrt{1+4\sigma^2} - 2\sigma)^{2k}-2 \geq (4k\sigma)^2,
$$
which completes the proof for \prettyref{eq:ak_Estimate}. To see this, define $g(x)=x^k-x^{-k}$ which is increasing in $x$. Then 
$$ 
(\sqrt{1+4\sigma^2} + 2\sigma)^{2k} + (\sqrt{1+4\sigma^2} - 2\sigma)^{2k}-2= g\left( \sqrt{1+4\sigma^2}  + 2\sigma \right)^2
\ge g\left( 1 + 2\sigma \right)^2 \ge (4k\sigma)^2,
$$
where the last inequality holds because $(1+a)^k - (1-a)^k \ge 2ak$ for $a\ge 0$. 
Finally, \prettyref{eq:MGF_Compute} follows from \prettyref{eq:MGF_ak}.
\end{proof}


Based on the above representation via cycle decomposition, we have the following estimate for the moment generating function. This estimate is a key result in this paper as it is the basis of both Theorem \ref{thm:main}
and Lemma \ref{lem:High_Prob_Event}.

\begin{lemma}\label{lem:Estimate_MGF}
Suppose $ d=o(\log n) $. For some $ \sigma_0>0 $, let $ \delta = \sigma_0/\sqrt{n} $ and $ N \subset O(d) $ be the $ \delta $-net defined in \prettyref{eq:net}.
\begin{enumerate}
\item[(i)] If $ \sigma_0 = o(n^{-2/d}) $, then
\begin{equation}
\sum_{\Pi \neq I_n} \sum_{Q \in N} \expect \exp \left\{ -\frac{1}{32 \sigma_0^2} \Fnorm{X-\Pi X Q}^2 \right\} = o(1).
\end{equation}
\item[(ii)] For any $ \eps=\eps(n)>0 $, if $ \sigma_0^{-d} > 16 n 2^{2/\eps}  $, then the following is true
\begin{equation}
\sum_{\diff(\pi,\mathrm{Id}) \geq \eps n} \sum_{Q \in N} \expect \exp \left\{ -\frac{1}{32 \sigma_0^2} \Fnorm{X-\Pi X Q}^2 \right\} = o(1).
\end{equation}
\end{enumerate}
\end{lemma}

\begin{proof}
(i) For any fixed $ \Pi \in \fS_n $,
combining \prettyref{eq:ak_Estimate} and \prettyref{eq:a1_Estimate} yields
\begin{align}
\prod_{k \geq 1}a_k(Q)^{n_k} &\leq (C \sigma_0)^{n_1 d + \sum_{k \geq 2}n_k(k-1)d} \pth{\prod_{\ell =1}^d \frac{1}{\theta_\ell +\sigma_0}}^{n_1} \nonumber \\
&= (C\sigma_0)^{d(n-\sum_{k \geq 2}n_k)} \pth{\prod_{\ell =1}^d \frac{1}{\theta_\ell +\sigma_0}}^{n_1} \nonumber \\
&\leq (C\sigma_0)^{\frac{n+n_1}{2}d} \pth{\prod_{\ell =1}^d \frac{1}{\theta_\ell +\sigma_0}}^{n_1}. \label{eq:ak_product}
\end{align}
Note that by Lemma \ref{lem:Net_Size}, we have
\begin{equation}
\left| N\pth{ e^{\ii \frac{m_1 \delta}{4}},\dots,e^{\ii \frac{m_d \delta}{4}} } \right| \leq \pth{1+\frac{\max |m_\ell|}{2}}^{2d^2} \leq \pth{1+\frac{\sum_{\ell=1}^d |m_\ell|}{2}}^{2d^2} \leq \prod_{\ell=1}^d \pth{1+\frac{|m_\ell|}{2}}^{2d^2}. \label{eq:Net_Size_Product}
\end{equation}
Using Lemma \ref{lem:MGF} and \prettyref{eq:ak_product}, this leads to
\begin{align*}
&~ \sum_{\Pi \neq I_n} \sum_{Q \in N} p(\Pi,Q)\\
&\leq \sum_{n_1=0}^{n-2} \sum_{m_1,\dots,m_d= \lfloor -\frac{4 \pi}{\delta} \rfloor }^{\lceil \frac{4 \pi}{\delta} \rceil} \left| N\pth{ e^{\ii \frac{m_1 \delta}{4}},\dots,e^{\ii \frac{m_d \delta}{4}} } \right| (n-n_1)! \binom{n}{n_1} (C\sigma_0)^{\frac{n+n_1}{2}d} \pth{\prod_{\ell =1}^d \frac{1}{\frac{\delta |m_\ell|}{4} +\sigma_0}}^{n_1}\\
&\leq \sum_{n_1=0}^{n-2} (C\sigma_0)^{\frac{n+n_1}{2}d} (n-n_1)! \binom{n}{n_1} \qth{ \sum_{m=\lfloor -\frac{4 \pi}{\delta} \rfloor}^{\lceil \frac{4 \pi}{\delta} \rceil} \frac{1}{(\frac{\delta |m|}{4} + \sigma_0)^{n_1}} (1+\tfrac{|m|}{2})^{2d^2} }^d\\
&\leq \sum_{n_1=0}^{n-2} (C\sigma_0)^{\frac{n-n_1}{2}d} (n-n_1)! \binom{n}{n_1} \qth{ \sum_{m=\lfloor -\frac{4 \pi}{\delta} \rfloor}^{\lceil \frac{4 \pi}{\delta} \rceil} \frac{1}{(1+\frac{\delta}{4\sigma_0}|m|)^{n_1}} (1+\tfrac{|m|}{2})^{2d^2} }^d\\
&\leq \sum_{n_1=0}^{n-2} \pth{(C\sigma_0)^d n^2}^{\frac{n-n_1}{2}} \qth{ \sum_{m=\lfloor -\frac{4 \pi}{\delta} \rfloor}^{\lceil \frac{4 \pi}{\delta} \rceil} \frac{1}{(1+\frac{\delta}{4\sigma_0}|m|)^{n_1}} (1+\tfrac{|m|}{2})^{2d^2} }^d,
\end{align*}
where the second line follows from Lemma \ref{lem:Net_Size} and the fourth line follows from the fact that the number of permutations with $ n_1 $ fixed points is at most $ (n-n_1)! \binom{n}{n_1} \leq n^{n-n_1} $.

Recall that $ \delta=\sigma_0/\sqrt{n} $ and $ \sigma_0 = o(n^{-2/d}) $. For any fixed $ 1 \leq n_1 \leq n-2 $, 
$$ T_{n_1} \triangleq \sum_{m=\lfloor -\frac{4 \pi}{\delta} \rfloor}^{\lceil \frac{4 \pi}{\delta} \rceil} \frac{1}{(1+\frac{\delta}{4\sigma_0}|m|)^{n_1}} (1+\tfrac{|m|}{2})^{2d^2} = \sum_{m=\lfloor -\frac{4 \pi}{\delta} \rfloor}^{\lceil \frac{4 \pi}{\delta} \rceil} \frac{1}{(1+\frac{|m|}{4\sqrt{n}})^{n_1}} (1+\tfrac{|m|}{2})^{2d^2}. $$
If $ n_1 \leq \sqrt{n} $, we have
\begin{equation}
T_{n_1}  \leq \sum_{m=\lfloor -\frac{4 \pi}{\delta} \rfloor}^{\lceil \frac{4 \pi}{\delta} \rceil} \pth{1+\frac{|m|}{2}}^{2d^2} \leq \frac{8\pi}{\delta} \pth{ 1+\frac{2\pi}{\delta} }^{2d^2} \leq 2 \pth{\frac{4 \pi \sqrt{n} }{\sigma_0}}^{2d^2+1}. \label{eq:Tn1_small}
\end{equation}
Therefore, let $ \sigma_0^{-d}=L $ and $ L=n^2 K $ where $ K \gg 1 $, then
\begin{align}
\sum_{n_1=0}^{\sqrt{n}} \pth{(C\sigma_0)^d n^2}^{\frac{n-n_1}{2}} T_{n_1}^d &\leq \sqrt{n} \qth{ 2 \pth{\frac{4 \pi \sqrt{n} }{\sigma_0}}^{2d^2+1} }^d \pth{(C\sigma_0)^d n^2}^{\frac{n-\sqrt{n}}{2}} \nonumber \\
&\leq C^{d^3} n^{2d^3} L^{2d^2+1} K^{-\frac{n-\sqrt{n}}{2}} \nonumber \\
&\leq C^{d^3} n^{2d^3} L^{3d^2} K^{-\frac{n}{3}} \nonumber \\
&\leq C^{d^3} n^{2d^3} \exp \pth{3d^2 \log (n^2 K)} \exp \pth{-\frac{n}{3} \log K} \nonumber \\
&\leq C^{d^3} \exp \pth{ (6d^2+2d^3) \log n - \pth{\frac{n}{3} - 3d^2} \log K } \nonumber \\
&=o(1), \label{eq:n1_small}
\end{align}
where the last line follows from $ K \gg 1 $ and $ d=o(\log n) $.

On the other hand, for $ \sqrt{n} \leq n_1 \leq n-2 $,we decompose it into two parts $ T_{n_1} = J_1+J_2 $, where
\begin{align*}
J_1 &\triangleq \sum_{|m| \leq 8\sqrt{n}} \frac{1}{(1+\frac{|m|}{4\sqrt{n}})^{n_1}} (1+\tfrac{|m|}{2})^{2d^2},\\
J_2 &\triangleq \sum_{8\sqrt{n} < |m| \leq \frac{4\pi}{\delta}} \frac{1}{(1+\frac{|m|}{4\sqrt{n}})^{n_1}} (1+\tfrac{|m|}{2})^{2d^2}.
\end{align*}
We first show that the contribution of $ J_2 $ is negligible. To see this, note that
\begin{align*}
J_2 &\leq C (4\sqrt{n})^{n_1} \sum_{m=1+8\sqrt{n}}^{4\pi / \delta} m^{-n_1 + 2d^2}\\
&\leq C (4\sqrt{n})^{n_1} \int_{8\sqrt{n}}^{4\pi / \delta} x^{-(n_1-2d^2)} \diff x\\
&\leq C (4\sqrt{n})^{n_1} \frac{1}{n_1 - 2d^2-1} (8 \sqrt{n})^{-n_1 +2d^2+1}\\
&\leq C 2^{-n_1 + 6d^2 + 3} \frac{1}{n_1 - 2d^2-1} n^{d^2+\frac{1}{2}}\\
&\leq C 2^{-n_1/2} n^{d^2}.
\end{align*}
Recall that $ n_1 \geq \sqrt{n} $ and $ d=o(\log n) $. Therefore we have $ J_2 = o(1) $. Moreover, a simple observation is that $ T_{n_1} \geq 1 $. This concludes that $ J_2 $ is negligible and it suffices to bound $ J_1 $. Note that for $ 0 \leq x \leq 2 $ we have $ 1+x \geq e^{x/2} $. Therefore, this implies
$$ J_1 \leq C \sum_{m=0}^{8\sqrt{n}} \exp \pth{- \pth{ \frac{n_1}{8 \sqrt{n}} -2d^2 } m}. $$
For $ n_1 \geq 32 \sqrt{n} (\log n)^2  $, we have $ \tfrac{n_1}{8 \sqrt{n}} -2d^2 > \tfrac{n_1}{16 \sqrt{n}} $ since $ d=o(\log n) $. Consequently, in this regime we have
\begin{equation*}
J_1 \leq C \sum_{m=0}^{8\sqrt{n}} \exp \pth{ - \frac{n_1}{16 \sqrt{n}} m } \leq \frac{C}{1-e^{ -\frac{n_1}{16\sqrt{n}} }} \leq \frac{C}{1-e^{-4(\log n)^2}}.
\end{equation*}
Thus, for $ n_1 \geq 32 \sqrt{n}(\log n)^2 $, we have
\begin{equation}
T_{n_1}^d \leq (2J_1)^d \leq C^d \pth{ 1-e^{-4(\log n)^2} }^{-d} \leq C^d \exp \pth{ d e^{-4 (\log n)^2} } \leq C^d. \label{eq:Tn1_large1}
\end{equation}
For $ \sqrt{n} \leq n_1 < 32\sqrt{n}(\log n)^2 $, we use a trivial bound
$$ J_1 \leq C \sum_{m=0}^{8 \sqrt{n}} \pth{1+\frac{m}{2}}^{2d^2} \leq C (8 \sqrt{n})^{2d^2+1}. $$
In this case,
\begin{equation}
T_{n_1}^d \leq C^d (8 \sqrt{n})^{2d^2+1}. \label{eq:Tn1_large2}
\end{equation}
Thus, \prettyref{eq:Tn1_large1} and \prettyref{eq:Tn1_large2} together imply
\begin{align}
&~ \sum_{n_1=\sqrt{n}}^{n-2} \pth{(C\sigma_0)^d n^2}^{\frac{n-n_1}{2}} T_{n_1}^d \nonumber \\
&\leq \sum_{n_1=\sqrt{n}}^{32\sqrt{n}(\log n)^2} \pth{(C\sigma_0)^d n^2}^{\frac{n-n_1}{2}} T_{n_1}^d + \sum_{n_1=32\sqrt{n}(\log n)^2}^{n-2} \pth{(C\sigma_0)^d n^2}^{\frac{n-n_1}{2}} T_{n_1}^d \nonumber \\
& \leq 32 \sqrt{n}(\log n)^2 C^{d^2} (8\sqrt{n})^{2d^3 + d} 2^{-n}+ C^d \sigma_0^d n^2 \nonumber \\
&=o(1) \label{eq:n1_large}
\end{align}
Combining \prettyref{eq:n1_small} and \prettyref{eq:n1_large} together, we obtain
$$ \sum_{\Pi \neq I_n} \sum_{Q \in N} p(\Pi,Q) =o(1), $$
which completes the proof.

\smallskip

(ii) Due to the stronger noise level, we need to be more careful in \prettyref{eq:ak_product}:
\begin{align}
\prod_{j \geq 1}a_k(Q)^{n_j} &\leq (C \sigma_0)^{n_1 d + \sum_{j \geq 2}n_j(j-1)d} \pth{\prod_{\ell =1}^d \frac{1}{|\theta_\ell| +\sigma_0}}^{n_1} \nonumber \\
&= (C \sigma_0)^{dn-d\sum_{j=1}^n n_j} \prod_{\ell=1}^d \frac{1}{(1+\frac{|\theta_\ell|}{\sigma_0})^{n_1}}. \label{eq:ak_product_new}
\end{align}
For simplicity, denote by $ k \triangleq \diff(\pi,\mathrm{Id})=n-n_1 $ the number of non-fixed points of $\pi$. Let $ \tilde{\pi} $ be the restriction of the permutation $ \pi \in S_n $ on its non-fixed points, which by definition is a derangement. Denote the number of cycles of a permutation $ \pi $ by $ \frakc(\pi) $. An observation is that $ \frakc(\pi)=\sum_{j=1}^n n_j = n_1 + \frakc(\tilde{\pi}) $. Then Lemma \ref{lem:MGF} and \prettyref{eq:ak_product_new} yield
\begin{multline*}
\sum_{\diff(\pi,\mathrm{Id}) \geq \eps n} \sum_{Q \in N} p(\Pi,Q)\\
\leq \sum_{k=\eps n}^{n} \binom{n}{k} \sum_{\tilde{\pi}\ {\rm derangement}} \sum_{m_1,\dots,m_d= \lfloor -\frac{4 \pi}{\delta} \rfloor }^{\lceil \frac{4 \pi}{\delta} \rceil} \left| N\pth{ e^{\ii \frac{m_1 \delta}{4}},\dots,e^{\ii \frac{m_d \delta}{4}} } \right| (C \sigma_0)^{d(k-\frakc(\tilde{\pi}))} \prod_{\ell=1}^d \frac{1}{(1+\frac{\delta |m_\ell|}{4 \sigma_0})^{n-k}}.
\end{multline*}
Denote $ L=\sigma_0^{-d} $. Using \prettyref{eq:Net_Size_Product} and rearranging the above inequality give us
\begin{multline}
\sum_{\diff(\pi,\mathrm{Id}) \geq \eps n} \sum_{Q \in N} p(\Pi,Q)\\ 
\leq \sum_{k=\eps n}^n \binom{n}{k} L^{-k} \sum_{\tilde{\pi}\ {\rm derangement}} L^{\frakc(\tilde{\pi})} \qth{ \sum_{m= \lfloor -\frac{4 \pi}{\delta} \rfloor }^{\lceil \frac{4 \pi}{\delta} \rceil} \frac{1}{(1+\frac{\delta}{4\sigma_0}|m|)^{n-k}} (1+\tfrac{|m|}{2})^{2d^2} }^d. \label{eq:MGF_Estimate_New}
\end{multline}
Note that
$$ \sum_{\tilde{\pi}\ {\rm derangement}} L^{\frakc(\tilde{\pi})} = k! \, \expect_{\tau} \qth{ L^{\frakc(\tau)} \mathbbm{1}_{\{\tau \ \mathrm{is \ a \ derangement}\}} }, $$
where the expectation $ \expect_{\tau} $ is taken for a uniformly random permutation $ \tau \in S_k $. To bound the above truncated generating function, 
recall that the generating function of $\frakc(\tau)$ is given by (see, e.g., \cite[Eq.~(39)]{flajolet2009analytic})
\begin{equation}
\Expect_\tau[L^{\frakc(\tau)}] =  \binom{L+k-1}{k} = \frac{L(L+1)\cdots (L+k-1)}{k!}.
\label{eq:MGF-randomcycle}
\end{equation}
Pick some  $ \alpha \in (0,1) $ to be determined later and obtain the following
\begin{multline*}
\expect_{\tau} \qth{ L^{\frakc(\tau)} \mathbbm{1}_{\{\tau \ \mathrm{is \ a \ derangement}\}} }  \leq \expect_{\tau} \qth{ L^{\frakc(\tau)} \mathbbm{1}_{\{\frakc(\tau) \leq k/2 \}} }\\
\leq \expect_{\tau} \qth{L^{\alpha \frakc(\tau) + (1-\alpha)\frac{k}{2}}} =L^{(1-\alpha)\frac{k}{2}} \expect_{\tau} \qth{L^{\alpha \frakc(\tau)}} =L^{(1-\alpha)\frac{k}{2}} \binom{L^\alpha+k-1}{k}.
\end{multline*}
 Choosing $ \alpha=\tfrac{\log k}{\log L} $, we have
\begin{equation}
\expect_{\tau} \qth{ L^{\frakc(\tau)} \mathbbm{1}_{\{\tau \ \mathrm{is \ a \ derangement}\}} } \leq \binom{2k-1}{k} \pth{\frac{L}{k}}^{k/2} \leq \pth{\frac{16L}{k}}^{k/2}. \label{eq:MGF_Cycle}
\end{equation}

Recall that
$$ T_{n-k} = \sum_{m= \lfloor -\frac{4 \pi}{\delta} \rfloor }^{\lceil \frac{4 \pi}{\delta} \rceil} \frac{1}{(1+\frac{\delta}{4\sigma_0}|m|)^{n-k}} (1+\tfrac{|m|}{2})^{2d^2}. $$
For $ k \leq n-\sqrt{n} $, each term $ T_{n-k} $ is bounded by \prettyref{eq:Tn1_large1} and \prettyref{eq:Tn1_large2}. On the other hand, if $ k \geq n-\sqrt{n} $, we control $ T_{n-k} $ via \prettyref{eq:Tn1_small}.
Here in the case of almost perfect recovery, combined with \prettyref{eq:MGF_Cycle}, the assumption on $ \sigma_0 $ yields a superexponentially decaying term in the summation \prettyref{eq:MGF_Estimate_New}.
Specifically, combined this with \prettyref{eq:MGF_Estimate_New} and \prettyref{eq:MGF_Cycle}, we obtain
$$ \sum_{\diff(\pi,\mathrm{Id}) \geq \eps n} \sum_{Q \in N} p(\Pi,Q) \leq J_1 + J_2, $$
where
\begin{align*}
J_1 & \triangleq C^d \sum_{k=\eps n}^{n-32 \sqrt{n} (\log n)^2} \binom{n}{k}L^{-k} k! \pth{\frac{16 L}{k}}^{k/2}, \\
J_2 & \triangleq C^{d^3} n^{2d^3} L^{2d^2+1} \sum_{k=n-32 \sqrt{n}(\log n)^2 +1 }^n \binom{n}{k}L^{-k} k! \pth{\frac{16 L}{k}}^{k/2}.
\end{align*}
Let $ L=nK $ where $ \tfrac{\eps}{2}\log \tfrac{K}{16} > \log 2  $. Recall that $ d=o(\log n) $. Then applying Stirling's approximation gives us
\begin{equation}
J_1 \leq C^d n 2^n \pth{\frac{16 n}{L}}^{\eps n/2} \leq C^d n \exp \pth{ n \log 2 - \frac{\eps n}{2} \log \frac{K}{16} } = o(1),
\label{eq:Almost_k_small}
\end{equation}
and
\begin{align}
J_2 &\leq C^{d^3} n^{2d^3+1} L^{2d^2+1} 2^n \pth{\frac{16 n}{L}}^{n/3} \nonumber \\
&\leq C^{d^3} n^{2d^3+1} \exp \qth{ (2d^2+1) \log n + (2d^2+1) \log K + n \log 2 - \frac{n}{3} \log \frac{K}{16} } = o(1).
\label{eq:Almost_k_large}
\end{align}
Combining \prettyref{eq:Almost_k_small} and \prettyref{eq:Almost_k_large} implies
$$ \sum_{\diff(\pi,\mathrm{Id}) \geq \eps n} \sum_{Q \in N} p(\Pi,Q) = o(1), $$
which completes the proof.
\end{proof}

The estimate of the moment generating functions results in the following lemma, which plays a crucial rule in the probability reduction estimate \prettyref{eq:Prob_Reduction}.

\begin{lemma}\label{lem:High_Prob_Event}
For some $ \sigma_0>0 $, let $ \delta = \sigma_0/\sqrt{n} $ and $ N $ be the $ \delta $-net defined in \prettyref{eq:net}.
\begin{enumerate}
\item[(i)] If $ \sigma_0 = o(n^{-2/d}) $, for any constant $ c>0 $, the following inequality is true with high probability
\begin{equation}
\min_{\Pi \neq I_n} \min_{Q \in N} \Fnorm{X-\Pi X Q} \geq c \sqrt{d} \sigma_0.
\end{equation}
\item[(ii)] For any $ \eps=\eps(n)>0 $, if $ \sigma_0^{-d} > 16 n 2^{2/\eps}  $, the following is true for any fixed constant $ c>0 $ 
with high probability
\begin{equation}
\min_{\diff(\pi,\mathrm{Id}) \geq \eps n} \min_{Q \in N} \Fnorm{X-\Pi X Q} \geq c \sqrt{d} \sigma_0.
\end{equation}
\end{enumerate}
\end{lemma}

\begin{proof}
(i) For fixed $ \Pi \neq I_n $ and $ Q \in N $, by the Chernoff bound, for every $ t \geq 0 $ we have
\begin{multline*}
\prob{ \Fnorm{X-\Pi X Q} < c \sqrt{d} \sigma_0 }\\
=\prob{ e^{-t \Fnorm{X-\Pi X Q}^2} > e^{-t c^2 d \sigma_0^2} } \leq e^{t c^2 d \sigma_0^2} \expect \exp \pth{ -t \Fnorm{X-\Pi X Q}^2 }.
\end{multline*}
Taking $ t=\frac{1}{32 \sigma_0^2} $, by the union bound we have
\begin{align*}
\prob{ \min_{\Pi \neq I_n} \min_{Q \in N} \Fnorm{X-\Pi X Q} \geq c \sqrt{d} \sigma_0 } &= 1-\prob{\exists \Pi \neq I_n, \exists Q \in N \ s.t.\  \Fnorm{X-\Pi X Q} \leq c\sigma_0}\\
&\geq 1- e^{\frac{c^2 d}{32}} \sum_{\Pi \neq I_d} \sum_{Q \in N} \expect \exp \left\{ -\frac{1}{32 \sigma_0^2} \Fnorm{X-\Pi X Q}^2 \right\}\\
&\geq 1-o(1),
\end{align*}
where the last step follows from Lemma \ref{lem:Estimate_MGF}.

\smallskip

(ii) The arguments are similar with Part (i). Using Chernoff bound and Lemma \ref{lem:Estimate_MGF}, we have
\begin{align*}
&~ \prob{ \min_{\diff(\pi,\mathrm{Id}) \geq \eps n} \min_{Q \in N} \Fnorm{X-\Pi X Q} \geq c \sqrt{d} \sigma_0 }\\
\geq &~ 1- e^{\frac{c^2 d}{32}} \sum_{\diff(\pi,\mathrm{Id}) \geq \eps n} \sum_{Q \in N} \expect \exp \left\{ -\frac{1}{32 \sigma_0^2} \Fnorm{X-\Pi X Q}^2 \right\}\\
\geq &~ 1-o(1),
\end{align*}
which completes the proof.
\end{proof}


\subsection{Proof of Theorem \ref{thm:main}}
\begin{proof}
(i) For $ \sigma \ll n^{-2/d} $, let $ \delta=\sigma/\sqrt{n} $ and let $ N $ be the $ \delta $-net in operator norm for $ O(d) $ defined in \prettyref{eq:net}. Applying Lemma \ref{lem:Net_Error}, we have
\begin{align*}
\prob{ \|X^\top \Pi^\top Y\|_* \geq \|X^\top Y\|_* } &\leq \prob{ \max_{Q \in O(d)} \Iprod{X^\top \Pi^\top Y}{Q} \geq \Iprod{X^\top Y}{I_d} }\\
& \leq \prob{ \max_{Q \in N} \Iprod{X^\top \Pi^\top Y}{Q} \geq (1-\delta^2)\Iprod{X^\top Y}{I_d} }.
\end{align*}

For fixed $ \Pi $ and $ Q $, we have
\begin{multline*}
\prob{ \Iprod{X^\top \Pi^\top Y}{Q} \geq (1-\delta^2)\Iprod{X^\top Y}{I_d} }\\
=\prob{ \sigma \Iprod{Z}{(1-\delta^2)X-\Pi X Q} \geq (1-\delta^2)\fnorm{X}^2 - \Iprod{X}{\Pi X Q} }.
\end{multline*}
Note that we have the following observations
$$ \Fnorm{X}^2 - \iprod{X}{\Pi X Q} = \frac{1}{2} \Fnorm{X-\Pi X Q}^2, $$
and
\begin{align*}
\Fnorm{(1-\delta^2)X - \Pi X Q}^2 &= (1-\delta^2)^2 \Fnorm{X}^2 + \Fnorm{X}^2 - 2(1-\delta^2) \iprod{X}{\Pi X Q}\\
&= (1-\delta^2) \Fnorm{X-\Pi X Q}^2 - \delta^4 \Fnorm{X}^2.
\end{align*}
Therefore,
\begin{align}
& \prob{ \Iprod{X^\top \Pi^\top Y}{Q} \geq (1-\delta^2)\Iprod{X^\top Y}{I_d} } \nonumber \\
= & \prob{ \sigma \calN\left( 0,(1-\delta^2)\Fnorm{X-\Pi X Q}^2 - \delta^4 \Fnorm{X}^2 \right) \geq \frac{1}{2} \Fnorm{X-\Pi X Q}^2 - \delta^2 \Fnorm{X}^2 } \nonumber \\
\leq & \prob{ \sigma \calN\left( 0,\Fnorm{X-\Pi X Q}^2 \right) \geq \frac{1}{2} \Fnorm{X-\Pi X Q}^2 - \delta^2 \Fnorm{X}^2 }. \label{eq:Prob_Before_Reduction}
\end{align}

Consider the following events
$$ \calE_1 \triangleq \left\{ cdn \leq \Fnorm{X}^2 \leq Cdn \right\},\ \ \calE_2 \triangleq \sth{ \min_{\Pi \neq I} \min_{Q \in N} \Fnorm{X-\Pi X Q} \geq C \sqrt{d} \sigma }. $$
It is well known that $ \prob{\calE_1}=1-o(1) $, and by Lemma \ref{lem:High_Prob_Event} we also have $ \prob{\calE_2}=1-o(1) $.
On the events $ \calE_1 $ and $ \calE_2 $, the previous estimate \prettyref{eq:Prob_Before_Reduction} for $ \Pi \neq I $ reduces to
\begin{multline}
\prob{ \Iprod{X^\top \Pi^\top Y}{Q} \geq (1-\delta^2)\Iprod{X^\top Y}{I_d}, \calE_1,\calE_2  }\\  \leq \prob{ \sigma \calN\left( 0,\Fnorm{X-\Pi X Q}^2 \right) \geq \frac{1}{4} \Fnorm{X-\Pi X Q}^2 } \leq \expect \exp \left\{ -\frac{1}{32 \sigma^2} \Fnorm{X-\Pi X Q}^2 \right\}. \label{eq:Prob_Reduction}
\end{multline}
By Lemma \ref{lem:Estimate_MGF}, the reduction \prettyref{eq:Prob_Reduction} and a union bound, we have
\begin{align*}
&~\prob{ \max_{\Pi \neq I} \|X^\top \Pi^\top Y\|_* \geq \|X^\top Y\|_*  }\\
\leq &~ \prob{ \max_{\Pi \neq I} \|X^\top \Pi^\top Y\|_* \geq \|X^\top Y\|_* , \calE_1, \calE_2 } + \prob{ \calE_1^c } + \prob{ \calE_2^c } \\ 
\leq &~ \prob{ \max_{\Pi \neq I_n} \max_{Q \in N} \Iprod{X^\top \Pi^\top Y}{Q} \geq (1-\delta^2)\Iprod{X^\top Y}{I_d} ,\calE_1,\calE_2 } + o(1)\\
\leq &~ \sum_{\Pi \neq I_n} \sum_{Q \in N} \prob{  \Iprod{X^\top \Pi^\top Y}{Q} \geq (1-\delta^2)\Iprod{X^\top Y}{I_d} , \calE_1,\calE_2 } + o(1)\\
\leq &~ \sum_{\Pi \neq I_n}\sum_{Q \in N}  \expect \exp \left\{ -\frac{1}{32 \sigma^2} \Fnorm{X-\Pi X Q}^2 \right\} + o(1)\\
= &~ o(1).
\end{align*}
This implies that the ground truth $ \Pi^*=I_n $ is the approximate MLE with probability $ 1-o(1) $, i.e.,
$$ \prob{ \argmax_{\Pi \in S_n} \|X^\top \Pi^\top Y\|_* = I_n } = 1-o(1), $$
which shows the success of perfect recovery with high probability.

\smallskip

(ii) The arguments are essentially the same as Part (i). For a sufficiently small $ \eps=\eps(n)>0 $, take $ \sigma^{-d} > 16 n 2^{2/\eps} $ and consider the event
$$ \calE_2' \triangleq \sth{ \min_{\diff(\pi,\mathrm{Id}) \geq \eps n} \min_{Q \in N} \Fnorm{X-\Pi X Q} \geq C \sqrt{d} \sigma }. $$
Then Lemma \ref{lem:High_Prob_Event} implies $ \prob{\calE_2'}=1-o(1) $. On the event $ \calE_1 $ and $ \calE_2' $, the reduction estimate for $ \Pi $ with $ \diff(\pi,\mathrm{Id}) \geq \eps n $ still holds
$$ \prob{ \Iprod{X^\top \Pi^\top Y}{Q} \geq (1-\delta^2)\Iprod{X^\top Y}{I_d}, \calE_1,\calE_2'  } \leq \expect \exp \left\{ -\frac{1}{32 \sigma^2} \Fnorm{X-\Pi X Q}^2 \right\}. $$
Combining this with Lemma \ref{lem:Estimate_MGF}, we have
\begin{multline*}
\prob{ \max_{\diff(\pi,\mathrm{Id}) \geq \eps n} \|X^\top \Pi^\top Y\|_* \geq \|X^\top Y\|_*  }\\
\leq \sum_{\diff(\pi,\mathrm{Id}) \geq \eps n} \sum_{Q \in N} \expect \exp \left\{ -\frac{1}{32 \sigma^2} \Fnorm{X-\Pi X Q}^2 \right\} + o(1) = o(1).
\end{multline*}
Thus,
$$ \prob{ \overlap(\hat{\pi}_{\mathrm{AML}},\pi^*) \geq 1-\eps  } = 1-o(1). $$
Taking $ \sigma \ll n^{-1/d} $ so that $\epsilon=o(1)$, this implies the desired \prettyref{eq:main2}.
\end{proof}

\section{Proof for the distance model}
	\label{app:distance}



In this section, we prove Theorem \ref{thm:distance}. Let $ \tX \triangleq (I-\bfF)X $, $ \tY \triangleq (I-\bfF)Y $ and $ \tZ \triangleq (I-\bfF)Z $. Recall that the approximate MLE for the distance model is given by \prettyref{eq:MLEapprox-distance}.
As in the proof of \prettyref{thm:main},  thanks to the orthogonal invariance of the nuclear norm $ \|\cdot\|_* $, we may assume $ \tA^{1/2} = \tX $ and $ \tB^{1/2}=\tY $ without loss of generality, so that
\begin{equation*}
\tilde{\Pi}_{\mathrm{AML}} = \arg\max_{\Pi \in \fS(n)} \|\tX^\top \Pi^\top \tY\|_*.
\end{equation*}

Following the arguments for the dot-product model in Appendix \ref{sec:positve}, a key step is to
extend the estimate for $p(\Pi,Q)$ in \prettyref{eq:MGF}
to the following MGF:
\begin{equation}
\tilde{p}(\Pi,Q) \triangleq \expect \exp \sth{ -\frac{1}{32 \sigma^2} \Fnorm{\tX-\Pi \tX Q}^2 },
\label{eq:MGF_Distance}
\end{equation}
where $ \Pi \in \fS_n $ and $ Q \in O(d) $.
The following lemma gives a comparison between the MGF for the distance model and that for the dot-product model defined in \prettyref{eq:MGF}, the latter of which was previous estimated in \prettyref{lem:MGF}.

\begin{lemma}\label{lem:MGF_Compare}
Fix a permutation matrix $ \Pi \in \fS_n $. For $Q \in O(d)$, denote by $ e^{\ii \theta_1},\dots,e^{\ii \theta_d} $ the eigenvalues of $ Q $, where $ \theta_1,\dots,\theta_d \in [-\pi,\pi] $.
Then
\begin{equation}
\tilde{p}(\Pi,Q) \leq p(\Pi,Q) \prod_{\ell=1}^d \pth{1+\frac{\theta_l^2}{16\sigma^2}}^{1/2}. 
\label{eq:MGF_Compare}
\end{equation}
\end{lemma}

\begin{proof}
Let $ t=\tfrac{1}{32 \sigma^2} $. Denote by $ \tilde{x}=\vecc(\tX) \in \R^{nd} $ the vectorization of $ \tX $ and recall that $ x=\vecc(X) \in \R^{nd} $ satisfies $ x \sim \calN(0,I_{nd}) $. Then
$$ \Fnorm{\tX - \Pi \tX Q}^2 = \norm{(I_{nd}-Q^\top \otimes \Pi) \tilde{x}}^2 = \norm{ (I_{nd}-Q^\top \otimes \Pi) (I_d \otimes (I_n-\bfF)) x }^2. $$
Denote $ \tH \triangleq (I_{nd}-Q^\top \otimes \Pi) (I_d \otimes (I_n-\bfF)) $, then
$$  \tilde{p}(\Pi,Q) = \expect \exp \pth{ -t x^\top \tH^\top \tH x } = \qth{ \det \pth{I+2t \tH^\top \tH} }^{-\frac{1}{2}}. $$
It suffices to compute the eigenvalues of $ \tH $. Recall that the spectrum of $ \Pi $ is given by \prettyref{eq:Spectrum_Pi}. We claim that the spectrum of $ \tH $ is the following multiset
\begin{equation}
\mathsf{Spec}(\tH) = \pth{\mathsf{Spec}(H) \backslash \{ 1-e^{-\ii \theta_\ell}:\ell=1,\dots,d \}} \cup \sth{ 0 \ \mbox{with multiplicity} \  d },
\label{eq:Spectrum_tH}
\end{equation}
where $ \mathsf{Spec}(H) $ is the spectrum of $ H $ defined in Lemma \ref{lem:MGF}, given by
$$ \mathsf{Spec}(H) = \sth{ 1-e^{-\ii \theta_\ell}\lambda_j : \lambda_j \in \mathsf{Spec}(\Pi),\, j=1,\dots,n,\, \ell=1,\dots,d }. $$

Now we prove \prettyref{eq:Spectrum_tH}. As shown in \prettyref{eq:Spectrum_Pi}, $ \Pi $ has eigenvalue $ 1 $ with multiplicity $ \frakc(\Pi) $, where $ \frakc(\Pi) $ denote the number of cycles. We denote these by $ \lambda_1=\dots=\lambda_{\frakc(\Pi)} =1$. Using the cycle decomposition and the block diagonal structure as in Lemma \ref{lem:MGF}, we know that the eigenvectors corresponding to $ \lambda_1,\dots,\lambda_{\frakc(\Pi)} $ are of the following form
$$ v_i=(0,\dots,0,1,\dots,1,0,\dots,0)^\top, \quad i=1,\ldots,\frakc(\Pi) $$
where the number of $ 1 $'s equals the length of the corresponding cycle. In particular, due to the block diagonal structure, the $ 1 $ blocks in $ v_i $'s do not overlap. Therefore, we know that the vector $ \tilde{v}_1 = \tfrac{1}{\sqrt{n}} \sum_{i=1}^{\frakc(\Pi)}v_i= \tfrac{1}{\sqrt{n}}\mathbf{1}=\tfrac{1}{\sqrt{n}}(1,\dots,1)^\top \in \R^n $ is in the eigenspace of $ 1 $. Using the Gram-Schmidt process, we can construct vectors $ \tilde{v}_2,\dots,\tilde{v}_{\frakc(\Pi)} $ such that $ \{\tilde{v}_i\}_{i=1}^n $ is a orthonormal basis of the eigenspace, i.e.
$$ \iprod{\tilde{v}_i}{\tilde{v}_j} = \delta_{ij},\ \ \mathsf{span}(\tilde{v}_1,\dots,\tilde{v}_{\frakc(\Pi)})=\mathsf{span}(v_1,\dots,v_{\frakc(\Pi)}). $$
Pick an arbitrary eigenvalue $ \mu $ of $ Q^\top $ with eigenvector $ w \in \R^d $, and also pick an arbitrary eigenvalue $ \lambda $ of $ \Pi $ with eigenvector $ v \in \R^n $. Based on the arguments above, if $ \lambda \neq \lambda_1 $, then $ v \perp \tilde{v}_1 $, and therefore
\begin{equation}
\tH(w \otimes v) = w \otimes (I-\bfF)v-(Q^\top w) \otimes \Pi (I-\bfF)v = w \otimes v - \mu w \otimes \lambda v  = (1-\mu \lambda)(w \otimes v).
\label{eq:Eigen_Distance_Same}
\end{equation}
For the eigenpair $ (\lambda_1,\tilde{v}_1) $, we have
\begin{equation}
\tH(w \otimes \tilde{v}_1) = w \otimes (I-\bfF)\tilde{v}_1-(Q^\top w) \otimes \Pi (I-\bfF)\tilde{v}_1 = w \otimes 0 - \mu w \otimes 0 = 0.
\label{eq:Eigen_Distance_Zero}
\end{equation}
Combining \prettyref{eq:Eigen_Distance_Same} and \prettyref{eq:Eigen_Distance_Zero}, we conclude that for $ \ell=1,\dots,d $ and $ j=2,\dots,n $, the eigenvalue $ 1-e^{-\ii \theta_\ell}\lambda_j $ of $ H $ remains to be an eigenvalue of $ \tH $, while the eigenvalues $ 1-e^{-\ii\theta_\ell}\lambda_1 = 1-e^{-\ii\theta_\ell} $ of $ H $ are replaced by $ 0 $ in the spectrum of $ \tH $. Hence we have shown \prettyref{eq:Spectrum_tH} is true.

Using \prettyref{eq:Spectrum_tH} and \prettyref{eq:MGF_Expansion}, we obtain
\begin{align*}
\tilde{p}(\Pi,Q) &= \prod_{j=2}^n \prod_{\ell=1}^d \pth{1+2t|1-e^{-\ii \theta_\ell} \lambda_j|^2}^{-1/2}\\
&= p(\Pi,Q) \prod_{\ell=1}^d \pth{1+2t|1-e^{-\ii \theta_\ell}|^2}^{1/2}\\
&=p(\Pi,Q) \prod_{\ell=1}^d \pth{1+2t(2-2 \cos \theta_\ell)}^{1/2}\\
&\leq p(\Pi,Q) \prod_{\ell=1}^d \pth{1+\frac{\theta_\ell^2}{16 \sigma^2}}^{1/2},
\end{align*}
which completes the proof.
\end{proof}

Applying Lemma \ref{lem:MGF_Compare}, 
the following lemma is the counterpart of Lemma \ref{lem:Estimate_MGF}.

\begin{lemma}\label{lem:Estimate_MGF_Distance}
Suppose $ d=o(\log n) $. For some $ \sigma_0>0 $, let $ \delta = \sigma_0/\sqrt{n} $ and $ N \subset O(d) $ be the $ \delta $-net defined in \prettyref{eq:net}.
\begin{enumerate}
\item[(i)] If $ \sigma_0 = o(n^{-2/d}) $, then
\begin{equation}
\sum_{\Pi \neq I_n} \sum_{Q \in N} \expect \exp \left\{ -\frac{1}{32 \sigma_0^2} \Fnorm{\tX-\Pi \tX Q}^2 \right\} = o(1).
\end{equation}
\item[(ii)] For any $ \eps=\eps(n)>0 $, if $ \sigma_0^{-d} > 16 n 2^{2/\eps}  $, then the following is true
\begin{equation}
\sum_{\diff(\pi,\mathrm{Id}) \geq \eps n} \sum_{Q \in N} \expect \exp \left\{ -\frac{1}{32 \sigma_0^2} \Fnorm{\tX-\Pi \tX Q}^2 \right\} = o(1).
\end{equation}
\end{enumerate}
\end{lemma}

\begin{proof}
(i) Similarly as in Lemma \ref{lem:Estimate_MGF} Part (i), using \eqref{eq:MGF_Compare} we have
\begin{align*}
&~ \sum_{\Pi \neq I_n} \sum_{Q \in N} \tilde{p}(\Pi,Q)\\
\leq &~ \sum_{n_1=0}^{n-2} \sum_{m_1,\dots,m_d= \lfloor -\frac{4 \pi}{\delta} \rfloor }^{\lceil \frac{4 \pi}{\delta} \rceil} \left\{ \left| N\pth{ e^{\ii \frac{m_1 \delta}{4}},\dots,e^{\ii \frac{m_d \delta}{4}} } \right| (n-n_1)! \binom{n}{n_1} (C\sigma_0)^{\frac{n+n_1}{2}d} \right.\\
&~ \qquad \qquad \times \left. \qth{ \prod_{\ell =1}^d \frac{1}{(\frac{\delta |m_\ell|}{4} +\sigma_0)^{n_1}} \pth{1+\frac{\delta^2 m_\ell^2}{256 \sigma_0^2}}^{\frac{1}{2}} } \right\} \\
\leq &~ \sum_{n_1=0}^{n-2} \pth{(C\sigma_0)^d n^2}^{\frac{n-n_1}{2}} \qth{ \sum_{m= \lfloor -\frac{4 \pi}{\delta} \rfloor }^{\lceil \frac{4 \pi}{\delta} \rceil} \frac{1}{(1+\frac{\delta}{4\sigma_0}|m|)^{n_1}} \pth{1+\frac{\delta^2 m^2}{256 \sigma_0^2}}^{\frac{1}{2}} \pth{1+\frac{|m|}{2}}^{2d^2} }^d\\
= &~ \sum_{n_1=0}^{n-2} \pth{(C\sigma_0)^d n^2}^{\frac{n-n_1}{2}} \qth{ \sum_{m= \lfloor -\frac{4 \pi}{\delta} \rfloor }^{\lceil \frac{4 \pi}{\delta} \rceil} \frac{1}{(1+\frac{|m|}{4\sqrt{n}})^{n_1}} \pth{1+\frac{m^2}{256 n}}^{\frac{1}{2}} \pth{1+\frac{|m|}{2}}^{2d^2} }^d\\
\leq &~ \sum_{n_1=0}^{n-2} \pth{(C\sigma_0)^d n^2}^{\frac{n-n_1}{2}} \qth{ \sum_{m= \lfloor -\frac{4 \pi}{\delta} \rfloor }^{\lceil \frac{4 \pi}{\delta} \rceil} \frac{1}{(1+\frac{|m|}{4\sqrt{n}})^{n_1}}  \pth{1+\frac{|m|}{2}}^{2d^2 + 1} }^d.
\end{align*}
Let
$$ \tT_{n_1} \triangleq \sum_{m= \lfloor -\frac{4 \pi}{\delta} \rfloor }^{\lceil \frac{4 \pi}{\delta} \rceil} \frac{1}{(1+\frac{|m|}{4\sqrt{n}})^{n_1}}  \pth{1+\frac{|m|}{2}}^{2d^2 + 1}. $$
Using the same arguments as in \prettyref{eq:Tn1_small}, \prettyref{eq:Tn1_large1} and \prettyref{eq:Tn1_large2}, $ \tT_{n_1} $ can be bounded by
\begin{equation}
\tT_{n_1}^d \leq \left\{
\begin{aligned}
& C^{d^3} n^{d^3 +d} L^{2d^2+2} & &\mbox{if } n_1 \leq \sqrt{n},\\
& C^d (8 \sqrt{n})^{2d^2+2} & &\mbox{if } \sqrt{n} < n_1 < 32\sqrt{n}(\log n)^2,\\
& C^d & &\mbox{if } n_1 \geq 32\sqrt{n}(\log n)^2,
\end{aligned}
\right.  \label{eq:tTn1_Estimate}
\end{equation}
where $ L=\sigma_0^{-d} $. Consequently, following the similar estimates in \prettyref{eq:n1_small} and \prettyref{eq:n1_large},
$$ \sum_{\Pi \neq I_n} \sum_{Q \in N} \tilde{p}(\Pi,Q) = o(1),  $$
which completes the proof.

\smallskip

(ii) Combined with \prettyref{eq:MGF_Compare}, using the same arguments as in Lemma \ref{lem:Estimate_MGF} Part (ii) yields
\begin{equation*}
\sum_{\diff(\pi,\mathrm{Id}) \geq \eps n} \sum_{Q \in N} \tilde{p}(\Pi,Q) \leq \sum_{k=\eps n}^n \binom{n}{k}L^{-k} k! \pth{\frac{16 L}{k}}^{k/2} \tT_{n-k}^d = \tJ_1 + \tJ_2
\end{equation*}
where
\begin{align*}
\tJ_1 & \triangleq \sum_{k=\eps n}^{n-32 \sqrt{n} (\log n)^2} \binom{n}{k}L^{-k} k! \pth{\frac{16 L}{k}}^{k/2} \tT_{n-k}^d, \\
\tJ_2 & \triangleq \sum_{k=n-32 \sqrt{n}(\log n)^2 +1 }^n \binom{n}{k}L^{-k} k! \pth{\frac{16 L}{k}}^{k/2} \tT_{n-k}^d.
\end{align*}
By \prettyref{eq:tTn1_Estimate}, these two term can be bounded in the same way as in \prettyref{eq:Almost_k_small} and \prettyref{eq:Almost_k_large}. Thus,
$$ \sum_{\diff(\pi,\mathrm{Id}) \geq \eps n} \sum_{Q \in N} \tilde{p}(\Pi,Q) = o(1), $$
which completes the proof.
\end{proof}


Lemma \ref{lem:Estimate_MGF_Distance} implies the following high probability estimates. The proof is the same as in Lemma \ref{lem:High_Prob_Event} via Chernoff bound and therefore we omit it here.

\begin{lemma}\label{lem:High_Prob_Event_Distance}
Suppose $ d=o(\log n) $. For some $ \sigma_0>0 $, let $ \delta = \sigma_0/\sqrt{n} $ and $ N \subset O(d) $ be the $ \delta $-net defined in \prettyref{eq:net}.
\begin{enumerate}
\item[(i)] If $ \sigma_0 = o(n^{-2/d}) $, for any constant $ c>0 $, the following inequality is true with high probability
\begin{equation}
\min_{\Pi \neq I_n} \min_{Q \in N} \Fnorm{\tX-\Pi \tX Q} \geq c \sqrt{d} \sigma_0.
\end{equation}
\item[(ii)] For any $ \eps=\eps(n)>0 $, if $ \sigma_0^{-d} > 16 n 2^{2/\eps}  $, the following is true for any fixed constant $ c>0 $ with high probability
\begin{equation}
\min_{\diff(\pi,\mathrm{Id}) \geq \eps n} \min_{Q \in N} \Fnorm{\tX-\Pi \tX Q} \geq c \sqrt{d} \sigma_0.
\end{equation}
\end{enumerate}
\end{lemma}

Now we are ready to prove \prettyref{thm:distance}
. Similarly as in the dot-product model (see the remark following Theorem \ref{thm:main}), for almost perfect recovery, we actually prove a stronger nonasymptotic bound: For all sufficiently small $ \eps $, if $ \sigma^{-d} > 16 n 2^{2/\eps} $, then $ \overlap(\tilde{\pi}_{\mathrm{AML}},\pi^*) \geq 1-\eps $ with high probability, which clearly implies Theorem \ref{thm:distance} by taking $ \sigma \ll n^{-1/d} $.

\smallskip

\begin{proof}[Proof of Theorem \ref{thm:distance}]
(i) Let $ N $ be the $ \delta $-net for $ O(d) $ defined in \prettyref{eq:net}. Following the same argument as in Theorem \ref{thm:main}
\begin{align*}
\prob{ \|\tX^\top \Pi^\top \tY\|_* \geq \|\tX^\top \tY\|_* } \leq \prob{ \max_{Q \in N} \Iprod{\tX^\top \Pi^\top \tY}{Q} \geq (1-\delta^2)\Iprod{\tX^\top \tY}{I_d} }.
\end{align*}
For fixed $ \Pi $ and $ Q $, we have
\begin{multline*}
\prob{ \Iprod{\tX^\top \Pi^\top \tY}{Q} \geq (1-\delta^2)\Iprod{\tX^\top \tY}{I_d} }\\
=\prob{ \sigma \Iprod{\tZ}{(1-\delta^2)\tX-\Pi \tX Q} \geq (1-\delta^2)\fnorm{\tX}^2 - \Iprod{\tX}{\Pi \tX Q} }. 
\end{multline*}
Since the entries of $ \tilde{Z} $ are not independent, we need to be more careful:
\begin{align*}
    \Iprod{\tZ}{(1-\delta^2)\tX-\Pi \tX Q} 
    & = \Iprod{(I-\bfF)Z}{(1-\delta^2)\tX-\Pi \tX Q} \\
    & = \Iprod{Z}{(I-\bfF)((1-\delta^2)\tX-\Pi \tX Q)} \\
    & = \Iprod{Z}{(1-\delta^2)\tX - \Pi \tX Q},
\end{align*} 
because $(I-\bfF)\tilde X=\tilde X$ and 
$I-\bfF$ commutes with any permutation matrix $\Pi$.
Therefore, similarly as in \prettyref{eq:Prob_Before_Reduction},
\begin{align}
& \prob{ \Iprod{\tX^\top \Pi^\top \tY}{Q} \geq (1-\delta^2)\Iprod{\tX^\top \tY}{I_d} } \nonumber \\
= & \prob{ \sigma \calN\left( 0,(1-\delta^2)\Fnorm{\tX-\Pi \tX Q}^2 - \delta^4 \Fnorm{\tX}^2 \right) \geq \frac{1}{2} \Fnorm{\tX-\Pi \tX Q}^2 - \delta^2 \Fnorm{\tX}^2 } \nonumber \\
\leq & \prob{ \sigma \calN\left( 0,\Fnorm{\tX-\Pi \tX Q}^2 \right) \geq \frac{1}{2} \Fnorm{\tX-\Pi \tX Q}^2 - \delta^2 \Fnorm{\tX}^2 }. \label{eq:Prob_Before_Reduction_Distance}
\end{align}

Consider the events
$$ \calE_1 \triangleq \left\{ cdn \leq \Fnorm{\tX}^2 \leq Cdn \right\},\ \ \calE_2 \triangleq \sth{ \min_{\Pi \neq I} \min_{Q \in N} \Fnorm{\tX-\Pi \tX Q} \geq C \sqrt{d} \sigma }. $$
We claim that $ \prob{\calE_1}=1-o(1) $. To see this, note that 
\begin{multline*}
\fnorm{\tX}^2 = \Iprod{(I-\bfF)X}{(I-\bfF)X} = \Iprod{X}{(I-\bfF)X}\\
= \Tr(X^\top (I-\bfF)X) = \sum_{i=1}^d \sum_{\alpha,\beta=1}^n X_{\alpha i}X_{\beta i}(I-F)_{\alpha \beta}.
\end{multline*}
For each $ i=1,\dots,d $, we have $ \sum_{\alpha,\beta=1}^n X_{\alpha i}X_{\beta i}(I-\bfF)_{\alpha \beta} = \mathsf{Col}_i(X)^\top (I-\bfF) \mathsf{Col}_i(X) $, where $ \mathsf{Col}_i(X) \sim \calN(0,I_n) $ is the $ i $-th column of $ X $.
By Hanson-Wright inequality (see e.g. \cite[Theorem 1.1]{RudVer13}), for each $ t \geq 0 $,
\begin{multline*}
\prob{ |\mathsf{Col}_i(X)^\top (I-\bfF) \mathsf{Col}_i(X) - \expect \mathsf{Col}_i(X)^\top (I-\bfF) \mathsf{Col}_i(X)| >t }\\
\leq 2 \exp \qth{ -c \min \pth{ \frac{t^2}{\Fnorm{I-\bfF}^2} , \frac{t}{\norm{I-\bfF}} } }.
\end{multline*}
Taking $ t=n^{3/4} $ and simplifying the above inequality yield
\begin{equation}
\prob{ |\mathsf{Col}_i(X)^\top (I-\bfF) \mathsf{Col}_i(X) - (n-1)| > n^{3/4} } \leq 2 \exp \pth{-c n^{1/2}}. \label{eq:tX_Norm_Concentration}
\end{equation}
Note that \prettyref{eq:tX_Norm_Concentration} is true for every $ i=1,\dots,d $, and the columns $ \mathsf{Col}_i(X) $'s are independent. This immediately gives us $ \prob{\calE_1} = 1-o(1) $. Moreover, by Lemma \ref{lem:High_Prob_Event_Distance} we also have $ \prob{\calE_2}=1-o(1) $. On the events $ \calE_1 $ and $ \calE_2 $, the estimate \prettyref{eq:Prob_Before_Reduction_Distance} reduces to
\begin{multline}
\prob{ \Iprod{\tX^\top \Pi^\top \tY}{Q} \geq (1-\delta^2)\Iprod{\tX^\top \tY}{I_d}, \calE_1,\calE_2  }\\ 
\leq \prob{ \sigma \calN\left( 0,\Fnorm{\tX-\Pi \tX Q}^2 \right) \geq \frac{1}{4} \Fnorm{\tX-\Pi \tX Q}^2 } 
\leq \expect \exp \left\{ -\frac{1}{32 \sigma^2} \Fnorm{\tX-\Pi \tX Q}^2 \right\}. \label{eq:Prob_Reduction_Distance}
\end{multline}
Combining this with Lemma \ref{lem:Estimate_MGF_Distance} and applying a union bound, we have
\begin{align*}
&~ \prob{ \max_{\Pi \neq I} \|\tX^\top \Pi^\top \tY\|_* \geq \|\tX^\top \tY\|_*  }\\
\leq &~ \prob{ \max_{\Pi \neq I_n} \max_{Q \in N} \Iprod{\tX^\top \Pi^\top \tY}{Q} \geq (1-\delta^2)\Iprod{\tX^\top \tY}{I_d} ,\calE_1,\calE_2 } + \prob{ \calE_1^c } + \prob{ \calE_2^c }\\
\leq &~ \sum_{\Pi \neq I_n} \sum_{Q \in N} \prob{  \Iprod{\tX^\top \Pi^\top \tY}{Q} \geq (1-\delta^2)\Iprod{\tX^\top \tY}{I_d} , \calE_1,\calE_2 } + o(1)\\
\leq &~ \sum_{\Pi \neq I_n}\sum_{Q \in N}  \expect \exp \left\{ -\frac{1}{32 \sigma^2} \Fnorm{\tX-\Pi \tX Q}^2 \right\} + o(1)\\
= &~ o(1).
\end{align*}
This implies $ \tilde{\pi}_{\mathrm{AML}} = \mathrm{Id} $ with high probability, which completes the proof.

\smallskip

(ii) The idea is the same as Theorem \ref{thm:main} Part (ii). For a sufficiently small $ \eps=\eps(n)>0 $, take $ \sigma^{-d} > 16 n 2^{2/\eps} $ and consider the event
$$ \calE_2' \triangleq \sth{ \min_{\diff(\pi,\mathrm{Id}) \geq \eps n} \min_{Q \in N} \Fnorm{\tX-\Pi \tX Q} \geq C \sqrt{d} \sigma }. $$
Then Lemma \ref{lem:High_Prob_Event_Distance} implies $ \prob{\calE_2'}=1-o(1) $. On the event $ \calE_1 $ and $ \calE_2' $, the reduction estimate \prettyref{eq:Prob_Before_Reduction_Distance} for $ \Pi $ with $ \diff(\pi,\mathrm{Id}) \geq \eps n $ still holds
$$ \prob{ \Iprod{\tX^\top \Pi^\top \tY}{Q} \geq (1-\delta^2)\Iprod{\tX^\top \tY}{I_d}, \calE_1,\calE_2'  } \leq \expect \exp \left\{ -\frac{1}{32 \sigma^2} \Fnorm{\tX-\Pi \tX Q}^2 \right\}. $$
Combined with Lemma \ref{lem:Estimate_MGF_Distance}, we have
\begin{multline*}
\prob{ \max_{\diff(\pi,\mathrm{Id}) \geq \eps n} \|\tX^\top \Pi^\top \tY\|_* \geq \|\tX^\top \tY\|_*  }\\
\leq \sum_{\diff(\pi,\mathrm{Id}) \geq \eps n} \sum_{Q \in N} \expect \exp \left\{ -\frac{1}{32 \sigma^2} \Fnorm{\tX-\Pi \tX Q}^2 \right\} + o(1) = o(1).
\end{multline*}
Thus,
$$ \prob{ \overlap(\tilde{\pi}_{\mathrm{AML}},\pi^*) \geq 1-\eps  } = 1-o(1), $$
which completes the proof.
\end{proof}


\section{Information-theoretic necessary conditions}
\label{app:lower_bounds}

\newcommand{\Pigood}{\mathbf{\Pi}_{\sf good}}
\newcommand{\Pibad}{\mathbf{\Pi}_{\sf bad}}

In this section, we derive necessary conditions for both almost perfect recovery and perfect recovery
for the linear assignment model~\prettyref{eq:model-LAP}. These conditions also hold for 
the weaker dot-product and distance models.

\subsection{Impossibility of almost perfect recovery}

We first derive a necessary condition for almost perfect recovery that holds for any $d$ via a simple mutual information argument.
Then we focus on the special case where $d$ is a constant and
give a much sharper analysis,
improving the necessary condition 
from $\sigma \le n^{-(1-o(1))/d}$ to $\sigma = o(n^{-1/d})$.
Note that achieving a vanishing recovery error in expectation is equivalent to that with high probability (see e.g.~\cite[Appendix A]{HajekWuXu_one_info_lim15}). Thus without loss of generality, we focus on the expected number of errors $\expect{ \diff \left( \pi^*, \hat{\pi} \right) }$
in this subsection.


\begin{proposition}\label{prop:impossiblity}
For any $\epsilon \in (0,1)$, if there exists an estimator $\hat{\pi}\equiv \hat{\pi}(X,Y)$ such that $\expect{ \diff \left( \pi^*, \hat{\pi} \right) } \le \epsilon n$, then 
\begin{align}
\frac{d}{2} \log \left( 1+ \frac{1}{\sigma^2} \right) - \left( 1-\epsilon \right) \log n + 1 + \frac{\log (n+1) }{n} \ge 0. \label{eq:MI_assumption}
\end{align}
\end{proposition}
The necessary condition~\prettyref{eq:MI_assumption} further specializes to:
\begin{itemize}
\item $d=o(\log n)$: 
\begin{align}
\sigma =O\left( n^{-(1-\epsilon)/d} \right). \label{eq:lb1}
\end{align}
This yields \prettyref{thm:opt}\eqref{opt2} and resolves  \cite[Conjecture 1.4, item 1]{kunisky2022strong} in the positive;
\item $d=\Theta(\log n)$:
$$
\sigma \le \frac{1-\epsilon+o(1)}{\sqrt{n^{2/d}-1}} ;
$$


\item $d=\omega(\log n)$: 
$$
\sigma \le \sqrt{ \frac{d}{2(1-\epsilon -o(1)) \log n}}.
$$ 
In this case, this necessary condition matches the sufficient condition of almost perfect recovery in~\cite[Theorem 1]{dai2020achievability}
and~\cite[Section A.2]{kunisky2022strong} up to $1+o(1)$ factor,
thereby determining the sharp  information-theoretic limit for the linear assignment model in high dimensions.
\end{itemize}
\begin{proof}
Since $\pi^* \to (X,Y) \to \hat{\pi}$ form a Markov chain, by the data processing inequality of mutual information, we have
\begin{align}
I \left(\pi^*;X,Y \right) \ge I \left(\pi^*; \hat{\pi} \right) 
=H(\pi^*) - H \left( \pi^* | \hat{\pi} \right).
 \label{eq:MI_1}
\end{align}
On the one hand, note that $H(\pi^*)= \log(n!) \ge n \log n - n$. Moreover, for any fixed realization of $\hat \pi$,
the number of $\pi^*$ such that 
$ \diff \left( \pi^*, \hat{\pi} \right)=\ell$ is $\binom{n}{\ell} !\ell \le n^\ell$, 
where $!\ell$ denotes the number of derangements of $\ell$ elements, given by 
\[
!\ell = \ell! \sum_{i=0}^\ell \frac{(-1)^i}{i!} = \qth{\frac{\ell!}{e} },
\]
and $[\cdot]$ denotes rounding to the nearest integer.
Therefore,
$$
H \left( \pi^* | \hat \pi,\diff \left( \pi^*, \hat{\pi} \right)  \right) \le  \expect{\diff \left( \pi^*, \hat{\pi} \right) } \log n 
\le \epsilon n \log n.
$$
Furthermore, $\diff \left( \pi^*, \hat{\pi}\right) $ takes values in $\{0,1, \ldots, n\}$. Thus from the chain rule,
\begin{align}
H ( \pi^* | \hat \pi) = H ( \diff \left( \pi^*, \hat{\pi}\right)  | \hat \pi) + H \left(  \pi^* | \hat \pi, \diff \left( \pi^*, \hat{\pi}\right) \right)
\le \log (n+1) + \epsilon n \log n.  \label{eq:MI_2}
\end{align}
On the other hand, the information provided by the observation $(X,Y)$ about $\pi^*$ satisfies
\begin{align}
I(\pi^*;X,Y)
= & ~ I\left(\Pi^* X;\Pi^* X + \sigma Z | X \right)  \nonumber \\
\stepa{\leq} & ~ \frac{nd}{2} \log\pth{1+\frac{\Expect[\|X\|^2]}{nd\sigma^2}}  \nonumber \\
= & ~ \frac{nd}{2} \log\pth{1+\frac{1}{\sigma^2}}, \label{eq:MI_3}
\end{align}
where $(a)$ follows from the Gaussian channel capacity formula and the fact that the 
mutual information in the Gaussian channel under a second moment constraint is maximized by the
Gaussian input distribution. 
Combining~\prettyref{eq:MI_1}--\prettyref{eq:MI_3}, we get that 
$$
 \frac{nd}{2} \log\pth{1+\frac{1}{\sigma^2}} \ge \left( 1-\epsilon \right) n \log n - n - \log (n+1),
$$
arriving at the desired necessary condition~\prettyref{eq:MI_assumption}. 
\end{proof}


While the negative result in~\prettyref{prop:impossiblity} holds for any $d$, 
the necessary condition~\prettyref{eq:MI_assumption} turns out to be loose for bounded $d$. 
The following result gives the optimal condition in this case.
%

\begin{theorem}\label{thm:impossibility}
Assume $\sigma = \sigma_0 n^{-1/d}$ for any constant $\sigma_0 \in (0,1/2)$. 
There exists a constant $\delta_0(\sigma_0,d)$ that only depends on $\sigma_0, d$ such that 
for any estimator $\hat \Pi$ and all sufficiently large $n$,
$$
\expect{\diff \left(\Pi^*, \hat \Pi \right)} \ge \delta_0 n.
$$
\end{theorem}
\prettyref{thm:impossibility} readily implies that for constant $d$, $\sigma=o(n^{-1/d})$ is necessary for achieving the almost perfect recovery, \ie, 
$\expect{\diff (\Pi^*, \hat \Pi )}=o(n)$.
To prove~\prettyref{thm:impossibility}, we follow the program in~\cite{DWXY21} of analyzing the posterior distribution. 
The likelihood function of $(X,Y)$ given $\Pi^*=\Pi$ is proportional to $\exp( - \frac{1}{2\sigma^2 } \fnorm{Y- \Pi X}^2 )$.
Therefore, conditional on $(X,Y)$, the posterior distribution of $\Pi^*$ is a Gibbs measure,
given by 
$$
\mu_{X,Y}(\Pi) = \frac{1}{Z(X,Y)} \exp \left(  L(\Pi) \right), \quad \text{ where } L(\Pi) = \frac{1}{\sigma^2} \iprod{\Pi X}{Y},
$$
and $Z(X,Y)$ is the normalization factor. 

As observed in~\cite[Section 3.1]{DWXY21},  in order to prove the impossibility of almost perfect recovery, it suffices to consider the estimator $\tilde{\Pi}$ which
is sampled from the posterior distribution $\mu_{X,Y}(\Pi)$. To see this, given any estimator $\hat{\Pi} \equiv \hat{\Pi}(X,Y)$,
$(\hat{\Pi}, \Pi^*) $ and $(\hat{\Pi}, \tilde{\Pi})$ are equal in law, and hence
\[
\Expect[\diff(\tilde \Pi,\Pi^*)] \leq  \Expect[\diff(\tilde \Pi,\hat \Pi)]+  \Expect[\diff(\Pi^*,\hat \Pi)]= 2 \Expect[\diff(\Pi^*,\hat \Pi)],
\]
which shows that $\tilde \Pi$ is optimal within a factor of two.
Thus it suffices to bound $\Expect[\diff(\tilde \Pi,\Pi^*)] $ from below.

To this end, fix some $\delta$ to be specified later and 
define the sets of good and bad solutions respectively as 
\begin{align*}
\Pigood = & ~ \{ \Pi \in \fS_n: \diff(\Pi,\Pi^*) <  \delta n\} , \\
\Pibad = & ~ \{ \Pi \in \fS_n: \diff(\Pi,\Pi^*)  \geq  \delta n \}.
\end{align*}
By the definition of $\tilde \Pi$, we have 
\[
\Expect[\diff(\tilde \Pi,\Pi^*)] \geq \delta  n \cdot \Expect[\mu_{X,Y}(\Pibad)].
\]
Next we show two key lemmas, which bound the posterior mass of $\Pigood $ and $\Pigood$
from above and below, respectively.

\begin{lemma}
\label{lmm:Pigood}	
Assume $\sigma = \sigma_0 n^{-1/d}$ for any constant $\sigma_0 \in (0,1/2)$. 
For any constant $\delta$ such that $\delta \le 16(2\sigma_0)^d$, 
with probability at least $1- 4\delta n e^{- \delta n/\log n}$, 
	\begin{equation}
	\frac{\mu_{X,Y}(\Pigood)}{\mu_{X,Y}(\Pi^*)} \leq 2 \left( \frac{16 e^2  (2\sigma_0)^d }{\delta} \right)^{\delta n} .
	\label{eq:Pigood}
	\end{equation}
\end{lemma}

\begin{lemma}
\label{lmm:Pibad}
Assume $\sigma =\sigma_0 n^{-1/d}$ for some constant $\sigma_0$. 
There exist constants $\delta_0(\sigma_0,d)$ and $c(\sigma_0, d)$ that only depend on $\sigma_0,d$
such that for all $\delta \le \delta_0$ and sufficiently large $n$, with probability at least $1/2- c/n$, 
		\begin{equation}
	\frac{\mu_{X,Y}(\Pibad)}{\mu_{X,Y}(\Pi^*)} \geq e^{\delta_0 n/2}.
	\label{eq:Pibad}
	\end{equation}
\end{lemma}

Given the above two lemmas, Theorem \ref{thm:impossibility} readily follows. 
Indeed, combining \prettyref{lmm:Pigood} and \prettyref{lmm:Pibad} and choosing 
$\delta$ such that $\delta \log (16 e^2  (2\sigma_0)^d / \delta) = \delta_0/4$
we get 
 $\mu_{X,Y}(\Pibad) \geq \frac{e^{\delta_0 n/4 }}{2+ e^{\delta_0 n/4 } }$ with probability at least $1/2- c / n-4\delta n e^{- \delta n/\log n}$, 
 which shows that $\Expect[\diff(\tilde \Pi,\Pi^*)] \gtrsim  \delta n$ as desired.

\subsection{Upper bounding the posterior mass of good permutations}
In this section, we prove \prettyref{lmm:Pigood} by a truncated first moment calculation. We need the following key auxiliary result. 

\begin{lemma}\label{lmm:MGF_approximate}
Assume that $n (2\sigma)^d \le 1$. Then for any $\ell \in [0,n]$, 
$$
\sum_{\Pi: \diff (\Pi, \Pi^*) = \ell} \expect{ \exp\left( - \frac{1}{8\sigma^2} \fnorm{\Pi X - \Pi^*X}^2 \ \right)  } 
\le \left( \frac{16 n^2 (2\sigma)^d }{\ell} \right)^{\ell/2}. 
$$
\end{lemma}
\begin{proof}
It follows from~\prettyref{eq:ak_Estimate} in  \prettyref{lem:MGF} that 
$$
\expect{ \exp\left( - \frac{1}{8\sigma^2} \fnorm{\Pi X - \Pi^*X}^2 \ \right)  }  \le 
\prod_{k=1}^n \left[\left(2\sigma \right)^{k-1}\right]^{dn_k}
\le \left( 2\sigma \right)^{d (\ell- \frakc(\tilde{\pi}) )},
$$
where $\ell=n-n_1$ is the number of non-fixed points, $\tilde{\pi}$ is the restriction of the permutation $\pi$ on its non-fixed points,
and $\frakc(\tilde{\pi})$ denotes the number of cycles of $\tilde{\pi}$. 
It follows that 
\begin{align*}
\sum_{\Pi: \diff (\Pi, \Pi^*) = \ell} \expect{ \exp\left( - \frac{1}{8\sigma^2} \fnorm{\Pi X - \Pi^*X}^2 \ \right)  }
&\le \binom{n}{\ell}\frac{ \ell ! }{ L^\ell }  \expect_{\tau} \qth{ L^{\frakc(\tau)} \mathbbm{1}_{\{\tau \ \mathrm{is \ a \ derangement}\}} } \\
& \le \left(\frac{n}{L}\right)^{\ell} \expect_{\tau} \qth{ L^{\frakc(\tau)} \mathbbm{1}_{\{\tau \ \mathrm{is \ a \ derangement}\}} } \\
& \le  \left( \frac{16n^2 }{\ell L} \right)^{\ell/2},
\end{align*}
where $L=(2\sigma)^{-d}$, the expectation $ \expect_{\tau} $ is taken for a uniformly random permutation $ \tau \in S_\ell $,
and the last inequality follows from~\prettyref{eq:MGF_Cycle}.
\end{proof}

\begin{proof}[Proof of~\prettyref{lmm:Pigood}]
Note that 
$$
\frac{\mu_{X,Y}(\Pigood)}{\mu_{X,Y}(\Pi^*)} = \sum_{\Pi \in \Pigood} e^{L(\Pi)-L(\Pi^*)} = R_1 + R_2, 
$$
where 
\begin{align*}
R_1 \triangleq   & ~\sum_{\Pi: \diff(\Pi,\Pi^*) < \beta n/\log n }  e^{L(\Pi)-L(\Pi^*)} \\
R_2 \triangleq   & ~ \sum_{\Pi: \frac{\beta n}{\log n} \leq \diff(\Pi,\Pi^*) < \delta n  } e^{L(\Pi)-L(\Pi^*)}
\end{align*}
for some $\beta$ to be specified.
Next we bound $R_1$ and $R_2$ separately.


First, the number of permutations $\Pi$ such that $ \Pi^{-1} \circ \Pi^*$ has $\ell$ non-fixed points is 
\begin{equation}
|\{\Pi \in \fS_n: \diff(\Pi,\Pi^*)=\ell\}| = !\ell \cdot \binom{n}{\ell},
\label{eq:mcount}
\end{equation}
where $!\ell = \qth{\frac{\ell!}{e} }$.
Thus
\begin{equation}
\frac{1}{2e} n(n-1)\cdots (n-\ell+1) \leq 
|\{ \Pi \in \fS_n: \diff(\Pi,\Pi^*)=\ell\}| \leq \frac{2}{e} n(n-1)\cdots (n-\ell+1).
\label{eq:mcount2}
\end{equation}

Furthermore, for any $\Pi$, 
\begin{align}
\expect{ e^{L(\Pi)-L(\Pi^*)} } &= \expect{\exp\left(  \frac{1}{\sigma^2} \iprod{\Pi X - \Pi^*X }{Y} \right)} \nonumber \\
& = \expect{ \exp\left(  \frac{1}{\sigma^2} \iprod{\Pi X- \Pi^*X}{\Pi^*X}  + \frac{1}{2\sigma^2} \fnorm{\Pi X - \Pi^*X}^2 \right) } \nonumber \\
&= 1, \label{eq:MGF1}
\end{align}
where the first equality holds due to $Y=\Pi^*X+ \sigma^2 Z$ and $\expect{\exp( \iprod{A}{Z})} = \exp(\fnorm{A}^2/2)$ and the second equality follows from 
$\iprod{\Pi X- \Pi^*X}{\Pi^*X} = - \frac{1}{2} \fnorm{\Pi X - \Pi^*X}^2$. 

To bound $R_1$, using \prettyref{eq:mcount2} and \prettyref{eq:MGF1} we have
\begin{align*}
\expect{R_1} 
& = \sum_{\diff(\Pi,\Pi^*) < \frac{\beta n}{\log n} } \expect{ e^{ L(\Pi) - L(\Pi^*) } }  
\leq \sum_{\ell < \frac{\beta n}{\log n} }  \frac{2}{e} n^\ell \leq \frac{2\beta n}{e \log n} \exp(\beta n).
\end{align*}
By Markov's inequality, 
\begin{equation}
\prob{R_1 \geq e^{2 \beta n}} \leq \frac{2n}{e} \exp(-\beta n).
\label{eq:R1}
\end{equation}

To bound $R_2$, the calculation above shows that directly applying the Markov inequality is too crude since $\Expect[R_2]= e^{\Theta(n \log n)}$.
Note that although $L(\Pi)-L(\Pi^*)$ is negatively biased, 
when $L(\Pi)-L(\Pi^*)$ is atypically large it results in an excessive contribution to the exponential moments. Thus we truncate on the following event:
$$
\calE \triangleq  \bigcap_{ \Pi:  \frac{\beta n}{\log n}  \leq \diff (\Pi, \Pi^*) < \delta n}  \left\{ L(\Pi)-L(\Pi^*) \le  \tau\left( \diff(\Pi,\Pi^*) \right) \right\}
$$
for some threshold $\tau(\ell)$ to be chosen.

Then for any $c'>0$, 
\begin{align}
& \prob{ R_2  \ge e^{c' n} }  \nonumber \\
& \le \prob{\calE^c} + \prob{ \{R_2  \ge e^{c' n}\}   \cap \calE } \nonumber \\
& \le \prob{\calE^c} +\prob{ \sum_{\frac{\beta n}{\log n} \leq \diff(\Pi,\Pi^*)  <  \delta n } e^{  L(\Pi)-L(\Pi^*) }  \indc{ L(\Pi)-L(\Pi^*) \leq \tau\left( \diff(\Pi,\Pi^*) \right) }  \ge e^{c' n} } \nonumber  \\
& \le  \prob{\calE^c} + e^{-c' n} \sum_{\frac{\beta n}{\log n} \leq \diff(\Pi,\Pi^*)  <  \delta n } \expect{e^{ L(\Pi)-L(\Pi^*) }  \indc{ L(\Pi)-L(\Pi^*) \leq \tau\left( \diff(\Pi,\Pi^*) \right) } }.  \label{eq:R2}
\end{align}

To bound the first term, note that for any $t>0$,
\begin{multline*}
\prob{ L(\Pi) - L(\Pi^*) \geq \tau } \\ \le e^{- t \tau } \expect{ \exp\left(  \frac{t}{\sigma^2} \iprod{\Pi X - \Pi^*X }{Y} \right) }
 =e^{- t \tau}  \expect{ \exp\left(  \frac{t^2 - t }{2\sigma^2} \fnorm{\Pi X - \Pi^*X}^2 \ \right)  }.
\end{multline*}
By choosing $t=1/2$, we get that 
\begin{align*}
\prob{ L(\Pi) - L(\Pi^*) \geq \tau } \le e^{-  \tau/2}  \expect{ \exp\left( - \frac{1}{8\sigma^2} \fnorm{\Pi X - \Pi^*X}^2 \ \right)  }.
\end{align*}
Recall from~\prettyref{lmm:MGF_approximate}, we have that
$$
\sum_{\Pi: \diff (\Pi, \Pi^*) = \ell} \expect{ \exp\left( - \frac{1}{8\sigma^2} \fnorm{\Pi X - \Pi^*X}^2 \ \right)  }
\le  \left( \frac{16 n^2 (2\sigma)^d }{\ell} \right)^{\ell/2} 
= \left( \frac{16 n (2\sigma_0)^d }{\ell} \right)^{\ell/2}.
$$
Therefore, it follows from a union bound  that
\begin{align}
\prob{\calE^c} & = \sum_{ \frac{\beta n}{\log n} \le \diff(\Pi,\Pi^*) <\delta n }  \prob{L(\Pi) - L(\Pi^*) \ge \tau\left( \diff(\Pi,\Pi^*) \right)  }  \nonumber   \\
& \le  \sum_{  \frac{\beta n}{\log n}  \le \ell < \delta n}  
e^{-  \tau(\ell)/2 }   \left( \frac{16 n (2\sigma_0)^d }{\ell} \right)^{\ell/2} \nonumber    \\
& =\sum_{  \frac{\beta n}{\log n}  \le \ell < \delta n}   e^{- \ell }  
\le  \delta n e^{ -\frac{\beta n}{\log n}  }, \label{eq:R2a}
\end{align}
where the last equality holds by choosing $\tau(\ell)= \ell \log (16 e^2 n (2\sigma_0)^d /\ell )$.


For the second term in \prettyref{eq:R2a}, we bound the truncated MGF as follows:
\begin{align*}
& \sum_{\Pi: \diff (\Pi, \Pi^*) = \ell }  \expect{ e^{ L(\Pi) - L(\Pi^*)  }  \indc{L(\Pi) - L(\Pi^*)  \le \tau\left( \diff(\Pi,\Pi^*) \right)  } } \\
&\le  \sum_{\Pi: \diff (\Pi, \Pi^*) = \ell} \expect{\exp \left( \frac{1}{2} \left( L(\Pi) - L(\Pi^*) + \tau(\ell ) \right) \right) }  \\
& \le     \sum_{\Pi: \diff (\Pi, \Pi^*) = \ell}  \expect{ \exp\left( - \frac{1}{8\sigma^2} \fnorm{\Pi X - \Pi^*X}^2 \ \right)  } e^{\tau(\ell) /2} \\
& \le \left( \frac{16 n (2\sigma_0)^d }{\ell} \right)^{\ell/2} e^{\tau(\ell) /2} \\
& \le  \left( \frac{16 e n (2\sigma_0)^d }{\ell} \right)^{\ell}.
\end{align*}
It follows that 
\begin{align*}
\sum_{\frac{\beta n}{\log n} \leq \diff(\Pi,\Pi^*)  <  \delta n } \expect{e^{  L(\Pi) - L(\Pi^*)  }  \indc{ L(\Pi) - L(\Pi^*)  \le  r \diff(\Pi,\Pi^*) } } 
& \le  \sum_{  \frac{\beta n}{\log n}  \le \ell < \delta n }    \left( \frac{16 e n (2\sigma_0)^d }{\ell} \right)^{\ell} \\
& \le \delta n   \left( \frac{16 e  (2\sigma_0)^d }{\delta} \right)^{\delta n},
\end{align*}
where the last inequality holds for all $\delta \le 16 (2\sigma_0)^d$. 
Choosing  $c' = \delta  \log (  16 e^2  (2\sigma_0)^d/ \delta)$, we get that 
\begin{equation}
 e^{-c' n}\sum_{\frac{\beta n}{\log n} \leq \diff(\Pi,\Pi^*)  <  \delta n } \expect{e^{ L(\Pi) - L(\Pi^*)  }  \indc{ L(\Pi) - L(\Pi^*)  \le  r \diff(\Pi,\Pi^*)} }  
 \le \delta n e^{-\delta n} 
\label{eq:R2b}
\end{equation}
Substituting \prettyref{eq:R2a} and \prettyref{eq:R2b} into \prettyref{eq:R2}, we get
$$
\prob{R_2 \ge  \left( \frac{16 e^2  (2\sigma_0)^d }{\delta} \right)^{\delta n}  } \le 2 \delta n e^{- \beta n/\log n}.
$$
Combining this with \prettyref{eq:R1} and upon choosing $\beta= \delta $, we have 
$$
\prob{R_1+R_2 \geq  2 \left( \frac{16 e^2  (2\sigma_0)^d }{\delta} \right)^{\delta n}  } \leq 4 \delta n e^{-\delta n/\log n}, 
$$ 
concluding the proof.
\end{proof}

\subsection{Lower bounding the posterior mass of bad permutations}

In this section, we prove  \prettyref{lmm:Pibad}. We aim to construct exponentially many bad permutations $\pi$ whose log likelihood $L(\pi)$ is no smaller than $L(\pi^*)$. It turns out that $L(\pi)-L(\pi^*)$ can be decomposed according to the orbit
decomposition of $(\pi^*)^{-1}\circ \pi$ as per~\prettyref{eq:L_diff_decomp}. 
Thus, following~\cite{DWXY21}, we look for vertex-disjoint orbits $O$ whose total lengths add up to $\Omega(n)$ and each of them is \emph{augmenting} in the sense that $\Delta(O) \ge 0$. 

 In the planted matching model with independent weights~\cite{DWXY21}, a great challenge lies in the fact that short augmenting orbits (even after taking their disjoint unions) are insufficient to meet the $\Omega(n)$ total length requirement. As a result, one has to search for long augmenting orbits of length $\Omega(n)$. However, due to the excessive correlations among long augmenting orbits, the second-moment calculation fundamentally fails. To overcome this challenge, \cite{DWXY21} invents a two-stage finding scheme which first finds many but short augmenting paths and then patches them together to form a long augmenting orbit using the so-called sprinkling idea. Fortunately, in our low-dimensional case of $d=\Theta(1)$, as also observed in~\cite{kunisky2022strong},  it suffices to look for augmenting $2$-orbits and take their disjoint unions. 
 More precisely, 
 the following lemma shows that there are $\Omega(n)$ vertex-disjoint augmenting $2$-orbits, from which we can easily extract exponentially many different unions of total length $\Omega(n)$. 
In contrast, to prove the failure of the MLE for almost perfect recovery in~\cite{kunisky2022strong}, a single union of $\Omega(n)$ vertex-disjoint augmenting $2$-orbits is sufficient.






\begin{lemma}\label{lmm:aug_cycle}
If $\sigma=\sigma_0 n^{-1/d}$, then 
there exist constants  $c(\sigma_0,d)$, $\delta_0(\sigma_0,d)$, and $n_0(\sigma_0,d)$ that only depend on $\sigma_0$ and $d$
such that for all $n \ge n_0$, with probability at least $1/2- c/n$, 
there are at least $\delta_0 n$ many vertex-disjoint augmenting $2$-orbits. 
\end{lemma}
This lemma is proved in~\cite[Section 4]{kunisky2022strong} using the so-called concentration-enhanced second-moment method. 
For completeness, here we provide a much simpler proof via the vanilla second-moment method combined with Tur\'an's theorem.
\begin{proof}
Let $I_{ij}$ denote the indicator that $(i,j)$ is an augmenting $2$-orbit
and $I=\sum_{i<j} I_{ij}$. 
To extract a collection of vertex-disjoint augmenting $2$-orbits, we construct a graph $G=(V, E)$, where 
the vertices correspond to $(i,j)$ for which $I_{ij}=1$, and  $(i,j)$ and $(k,\ell)$ are connected
if $(i,j)$ and $(k,\ell)$ share a common vertex. By construction, any collection of vertex-disjoint $2$-orbits
corresponds to an independent set in $G$. 
By Tur\'an's theorem (see e.g.~\cite[Theorem 1, p.~95]{Alon2016}), 
there exists an independent set $S$ in $G$ of size at least 
$|V|^2 / ( 2|E| + |V|) $. 
It remains to bound $|V|$ from below and $|E|$ from above.

Note that $|V|=I = \sum_{i<j} I_{ij}$. For all $n$ sufficiently large, $\sigma^2 \le d/40$ and 
it follows from~\cite[Prop.\ 4.3]{kunisky2022strong} that 
$$
 p  \triangleq  \prob{I_{ij}=1}  \ge   \frac{1}{1000 \sqrt{d}}  \left( 1+ \frac{1}{\sigma^2} \right)^{-d/2}. 
$$
Therefore,
\begin{align}
\expect{I} = \sum_{i<j} \prob{I_{ij}=1}  \ge \binom{n}{2} \frac{1}{1000\sqrt{d}}  \left( 1+ \frac{1}{\sigma^2} \right)^{-d/2} . \label{eq:mean_I}
\end{align}
Under the assumption that $\sigma=\sigma_0 n^{-1/d}$, it follows that $\expect{I} \ge c_0(d, \sigma_0) n$ 
for some constant $c_0(d,\sigma_0)$ that only depends on $d$ and $\sigma_0$. 
Moreover, 
\begin{align*}
\var(I) &=\sum_{i<j, k<\ell} \Cov\left(I_{ij}, I_{k\ell} \right)  \\
& = \sum_{i<j} \var(I_{ij}) + \sum_{i<j} \sum_{k: k\neq i,j}  \left( \Cov\left(I_{ij}, I_{ik} \right) + \Cov\left(I_{ij}, I_{jk} \right) \right)  \\
& \le \sum_{i<j} \expect{ I^2_{ij} } + \sum_{i<j} \sum_{k: k\neq i,j} \left( \expect{I_{ij}I_{ik}} + \expect{I_{ij}I_{jk}}\right),
\end{align*}
where the second equality holds because $I_{ij}$ and $I_{k\ell}$ are independent when 
$\{i,j\} \cap \{k,\ell\} =\emptyset$. 
Recall that $\expect{ I^2_{ij} }=\expect{ I_{ij}} = p$. 
Moreover, it follows from~\cite[Prop.\ 4.5]{kunisky2022strong} that 
$$
\expect{I_{ij}I_{ik}} \le 
\left(1+ \frac{3}{4\sigma^2} \right)^{-d}. 
$$
Combining the last three displayed equation yields that 
\begin{align}
\var(I)  \le \expect{I} + n^3 \left(1+ \frac{3}{4\sigma^2} \right)^{-d}.  \label{eq:variance_I}
\end{align}
Under the assumption that $\sigma=\sigma_0 n^{-1/d}$, 
it follows that $\var(I) \leq \expect{I}+c_1(d,\sigma_0)n $ 
for some $c_1(d,\sigma_0)$ that only depends on $d$ and $\sigma_0$. 
By Chebyshev's inequality, 
$$
\prob{ I \le \frac{1}{2} \expect{I} } \le \frac{4 \var(I) }{ \left(\expect{I}\right)^2  } 
\le \frac{4 (c_0+c_1)}{ c_0^2 n}. 
$$
Moreover, 
$$
|E| = \sum_{i<j} \sum_{k: k\neq i,j}  \left( I_{ij} I_{ik} + I_{ij} I_{jk} \right) 
$$
and hence
$$
\expect{|E|} = \sum_{i<j} \sum_{k: k\neq i,j} \left( \expect{I_{ij}I_{ik}} + \expect{I_{ij}I_{jk}}\right)
\le n^3\left(1+ \frac{3}{4\sigma^2} \right)^{-d}  \le  c_1(d, \sigma_0) n.
$$
By Markov's inequality, $|E| \le 2 \expect{|E|}$ with probability at least $1/2$. 
Therefore, with probability at least $1/2- 4 c_1 / (c_0^2 n)$, 
$$
|S| \ge \frac{ |V|^2 }{  2|E| + |V| } \ge  \frac{ \left(\expect{I}\right)^2/4 }{ 4 \expect{|E|} + \expect{I}/2 }
\ge \frac{ c_0^2 n^2 /4 }{ 4 c_1 n + c_0 n /2} \ge \delta_0 n,
$$
for some constant $\delta_0(d,\sigma_0)$ that only depends on $d$ and $\sigma_0$. 
\end{proof}
\begin{proof}[Proof of \prettyref{lmm:Pibad}]
By \prettyref{lmm:aug_cycle}, from $\delta_0 n$ such vertex-disjoint augmenting $2$-orbits, we choose $\delta_0 n/2$ many of them and  form a union of 
augmenting $2$-orbits with the total length $\delta_0 n /2 \times 2 = \delta_0 n$. There are $\binom{\delta_0 n}{\delta_0 n/2} $ many different unions,
and each of such union corresponds to a permutation $\Pi$ with $\diff(\Pi, \Pi^*)=\delta_0 n$ and $L(\Pi)\ge L(\Pi^*)$ in view of~\prettyref{eq:L_diff_decomp}. 
Therefore, for any $\delta \le \delta_0$, 
$$
\frac{\mu_{X,Y}(\Pibad)}{\mu_{X,Y}(\Pi^*)} \geq \binom{\delta_0 n}{\delta_0 n/2} \geq 2^{\delta_0 n/2}. 
$$
\end{proof}

\subsection{Impossibility of perfect recovery}
In this section, we prove an impossibility condition of perfect recovery. 
\begin{theorem}\label{thm:exact_nec}
Suppose that $\sigma^2 \le d/40$ and
\begin{align}
\frac{d}{4} \log \left( 1+ \frac{1}{\sigma^2} \right) -  \log n +  \log d \le C, \label{eq:exact_impossibility}
\end{align}
for a constant $C>0$. Then there exists a constant $c$ that only depends on $C$ such that for any estimator $\hat{\pi}$, 
$\prob{\hat{\pi} \neq \pi^*} \ge c$. 
\end{theorem}
\prettyref{thm:exact_nec} immediately implies that if there exists an estimator that achieves perfect recovery with high probability, then 
\begin{align}
\frac{d}{4} \log \left( 1+ \frac{1}{\sigma^2} \right) -  \log n +  \log d \to +\infty. 
\label{eq:exact_nec}
\end{align}
In comparison, it is shown in~\cite[Theorem 1]{dai2019database} that perfect recovery is possible if $\frac{d}{4} \log \left( 1+ \frac{1}{\sigma^2} \right) -  \log n \to +\infty$. Thus our necessary condition agrees with their sufficient condition up to an additive $\log d$ factor.
Our necessary condition~\prettyref{eq:exact_nec} further specializes to 
\begin{itemize}
\item $d \ll \log n$: 
\begin{align*}
\sigma \le 
\begin{cases} 
o(n^{-2/d}) & \text{ if } d=O(1) \\
n^{-2/d} & \text{ if } d \gg 1
\end{cases}.
\end{align*}
This yields \prettyref{thm:opt}\eqref{opt1} and slightly improves over the necessary condition of MLE in~\cite[Theorem 1.1]{kunisky2022strong}, that is, $\sigma =O(n^{-2/d})$.
\item $d=\Theta(\log n)$:
$$
\sigma \le \frac{1}{\sqrt{n^{4/d} -1 }};
$$
\item $d \gg \log n$: 
$$
\sigma \le  \sqrt{ \frac{d}{4\log (n/d) +\omega(1) } }.
$$
\end{itemize}
Note that the previous work~\cite{dai2019database}  shows that 
$\frac{d}{4} \log \left( 1+ \frac{1}{\sigma^2} \right) \ge (1-\Omega(1)) \log n$ is necessary for perfect recovery, under the additional assumption that $1 \ll d =O(\log n)$. The analysis therein  is based on showing the existence of an augmenting $2$-orbit via the second-moment method. 
We follow a similar strategy, but our first and second moment estimates are sharper and thus yield a tighter condition. 

\begin{proof}
Recall that $I_{ij}$ denote the indicator that $(i,j)$ is an augmenting $2$-orbit
and $I=\sum_{i<j} I_{ij}$. 
For the purpose of lower bound, consider the Bayesian setting where $\pi^*$ is drawn uniformly at
random. Then the MLE $\hat{\pi}_{\rm ML}$ given in~\prettyref{eq:LAP} minimizes the probability of error.
Hence,
it suffices to bound from below $\prob{\hat{\pi}_{\rm ML} \neq \pi^*}$. Note that on the event $\{I >0\}$, 
there exists at least one permutation $\pi \neq \pi^*$ whose likelihood is at least as large as that of $\pi^*$ and hence
the error probability of MLE is at least $1/2$. Therefore, 
$$
\prob{\hat{\pi}_{\rm ML} \neq \pi^*}  \ge \frac{1}{2} \prob{I >0}.
$$
It remains to bound $\prob{I >0}$ from below. To this end, we first bound $ \var(I)/ \left(\expect{I}\right)^2 $. 
In view of~\prettyref{eq:variance_I}, 
$$
 \frac{\var(I) }{\left(\expect{I}\right)^2 } \le \frac{1}{\expect{I}} + \frac{1}{\left(\expect{I}\right)^2 } n^3 \left(1+ \frac{3}{4\sigma^2} \right)^{-d}.
$$ 
By assumption $\sigma^2 \le d/40$ and~\prettyref{eq:exact_impossibility}, 
it follows from~\prettyref{eq:mean_I} that
$$
\expect{I}  \gtrsim  \frac{n^2}{\sqrt{d}}  \left( 1+ \frac{1}{\sigma^2} \right)^{-d/2}  \ge \exp \left( \frac{3}{2} \log d -2 C\right)
\ge \exp \left( -2 C\right). 
$$
Moreover, 
$$
\frac{1}{\left(\expect{I}\right)^2 } n^3 \left(1+ \frac{3}{4\sigma^2} \right)^{-d}
\lesssim \frac{d}{n} \left( \frac{1+1/\sigma^2}{1+3/(4\sigma^2)} \right)^d
\overset{(a)}{\le} \frac{d}{n} \left( 1+ \frac{1}{\sigma^2} \right)^{d/4} \overset{(b)}{\le} e^C,
$$
where $(a)$ holds because $1+3x/4 \ge (1+x)^{3/4}$ for all $x \ge 0$
and $(b)$ holds due to assumption~\prettyref{eq:exact_impossibility},
Combining the last three displayed equation yields that $ \var(I)/ \left(\expect{I}\right)^2  \le c_0$ for some constant $c_0$
that only depends on $C$. 
By the Paley-Zygmund inequality, 
$$
\prob{I >0 } \ge \prob{I \ge \frac{1}{2} \expect{I}} \ge \frac{\left(\expect{I}\right)^2 }{4 \left( \var(I) + \left(\expect{I}\right)^2 \right) }  
\ge \frac{1}{4c_0+1}. 
$$
\end{proof}

\section{Recovery thresholds in the nonisotropic case}
\label{app:nonisotropic}

In this section we argue that \prettyref{thm:main} continues to hold under the same conditions in the nonisotropic case of $X_i\iiddistr \calN(0,\Sigma)$, provided that $ \Sigma \succ c I $ for some absolute constant $ c>0 $.
In the general nonisotropic case, we denote by $ p(\Sigma,\sigma,\Pi,Q) $ the moment generating function given by \prettyref{eq:MGF} to highlight the dependency on the covariance matrix $ \Sigma $ and the noise level $ \sigma $. As in the proof of Lemma \ref{lem:Estimate_MGF}, recall $ x=\vecc(X) $ denotes the vectorization of $ X $. Since $ X_i \iiddistr \calN(0,\Sigma) $, the vector $ x \in \R^{nd} $ has distribution $ x \sim \calN(0,I_n \otimes \Sigma) $. Note that $ I_n \otimes \Sigma \succ c I_{nd} $.
Modifying \prettyref{eq:MGF_Determinant} accordingly, we have
\begin{multline*}
p(\Sigma,\sigma,\Pi,Q) = \expect \exp \pth{ -\frac{1}{32 \sigma^2} x^\top H^\top H x } = \qth{ \det \pth{I+ \frac{1}{16 \sigma^2} H^\top H (I_n \otimes \Sigma)} }^{-\frac{1}{2}}\\
\leq \qth{ \det \pth{I+ \frac{c}{16 \sigma^2} H^\top H} }^{-\frac{1}{2}} = p(I,\sigma',\Pi,Q), 
\end{multline*}
where $ H=I_{nd} - Q^\top \otimes \Pi $ and $ \sigma'=\sigma/\sqrt{c} $. This shows that the MGF $ p(\Sigma,\sigma,\Pi,Q) $ satisfies the same estimates \prettyref{eq:ak_Estimate}, \prettyref{eq:a1_Estimate} and Lemma \ref{lem:Estimate_MGF} for the isotropic case with the original noise $ \sigma $ replaced by a constant multiple of it $ \sigma' $. This constant multiplicative factor keeps $ \sigma' $ satisfying the same noise threshold in Theorem \ref{thm:main}, which implies both prefect recovery and almost perfect recovery can still be achieved for the nonisotropic case under the same conditions, hence confirming our claim in \prettyref{sec:discuss}.

\section*{Acknowledgment}
The authors are grateful to Zhou Fan, Cheng Mao, and Dana Yang for helpful discussions.

Y.~Wu is supported in part by the NSF Grant CCF-1900507, an NSF CAREER award CCF-1651588, and an Alfred Sloan fellowship.
J.~Xu is supported in part by the NSF Grant CCF-1856424 and an NSF CAREER award CCF-2144593.

\bibliography{matching}
\bibliographystyle{alpha}
\end{document}